\documentclass[11pt,a4paper]{article}
\usepackage{authblk}
\usepackage[utf8]{inputenc}
\usepackage[T1]{fontenc}
\usepackage{graphicx}
\graphicspath{{Figures_54/}}
\usepackage{overpic}
\usepackage{float}
\usepackage{subfigure}
\usepackage{multirow}
\usepackage{amssymb}
\usepackage{amsmath}
\usepackage{amsthm}

\usepackage[margin=1in]{geometry}
\usepackage{amsfonts}

\usepackage[numbers,sort&compress]{natbib}
\usepackage{hyperref}
\usepackage{caption}
\usepackage{mathrsfs}
\usepackage{xcolor}
\allowdisplaybreaks
\newtheorem{thm}{Theorem}[section]
\newtheorem{Lem}[thm]{Lemma}
\newtheorem{prop}[thm]{Proposition}
\newtheorem{cor}[thm]{Corollary}
\theoremstyle{definition}
\newtheorem{deff}[thm]{Definition}
\newtheorem{ex}[thm]{Example}

\title{Stability Regions and Bifurcations for Higher-Order Fractional Difference Equations}
\author{Janardhan Chevala}
\author{Sachin Bhalekar\footnote{Corresponding author: sachinbhalekar@uohyd.ac.in; email: janardhanch1992@gmail.com (J.C.).}}

\affil{School of Mathematics and Statistics, University of Hyderabad, Hyderabad 500046, India}

\date{}
\begin{document}
	
	\maketitle
\textit{This paper is dedicated to the memory of Prof. Ravi P. Agarwal.}
\begin{abstract}
We study stability regions for the higher-order, two-term fractional difference equation $\Delta^{\alpha}x(t) + a\,\Delta^{\beta}x(t + \alpha - \beta - 1) = (b - 1)x(t + \alpha - 2)$, where $0 < \beta \leq 1 < \alpha \leq 2$, $a > 0$, and $b \in \mathbb{C}$. The Z-transform yields a characteristic function whose image of the unit circle determines the stability boundary. Using a winding-number formulation, we give a necessary and sufficient root-count condition for asymptotic stability. Two analytically derived parameter values, $a_1 = 2^{\alpha-\beta}$ and $a_2 = 2^{\alpha-\beta}(4 - \alpha)/(2 - \beta)$, characterize an endpoint collision and a loss of regularity of the boundary curve, respectively. The real-parameter case and a nonlinear higher-order logistic map are treated as consequences of the same stability criterion. We also analyze the one-term family $\Delta^{\alpha}x(t) = (c - 1)x(t + \alpha - N)$ for $N - 1 < \alpha \leq N$. A winding-number bound proves that its stability region is empty for every $N \geq 3$. Numerical experiments illustrate the theoretical results.
\end{abstract}	


\section{Introduction} 
Difference equations provide a natural framework for processes whose states are observed or updated at discrete times \cite{kelley2001difference,elaydi2005introduction}. Their qualitative theory includes existence, oscillation, periodicity, and stability questions for both scalar equations and systems \cite{agarwal2000difference,agarwal2013advanced}. A higher-order scalar equation can be rewritten as a finite-dimensional first-order system, but its local behavior is governed by a characteristic polynomial of corresponding degree. Consequently, increasing the order can introduce additional characteristic roots, stability boundaries, and bifurcation mechanisms that are absent from first-order maps \cite{kocic1993global,elaydi2005introduction}.

\par Fractional difference operators extend discrete calculus by assigning nonlocal weights to past states and thereby incorporating memory directly into the recurrence relation \cite{miller1993introduction,atici2009discrete,mozyrska2015z,ferreira2022discrete}. Edelman's $\alpha$-families of maps and fractional-difference Caputo families provide influential frameworks for nonlinear discrete systems with power-law and falling-factorial-law memory \cite{edelman2013fractional,edelman2015fractional}. This memory changes the classical root conditions because the characteristic functions contain noninteger powers and branch points rather than ordinary polynomials. Stability regions for single-order fractional difference equations have been derived by spectral, Z-transform, and complex-variable methods \cite{stanislawski2013stability,abu2013asymptotic,vcermak2015explicit,wei2024explicit}. Related work has treated complex fractional orders, periodically varying maps, and discrete delays \cite{bhalekar2022stability,bhalekar2023fractional,joshi2024stability}. Two-term nabla fractional equations were analyzed through Volterra representations in \cite{cermak2012stability}, while a two-term forward-difference model with both orders below one was considered in \cite{chevala2025stability}.

\par The present problem is motivated by discrete processes in which two qualitatively different memory mechanisms act simultaneously. An operator of order $\alpha \in (1,2]$ represents a higher-order evolution with an inertial or multistep character, whereas an operator of order $\beta \in (0,1]$ introduces an additional, slower hereditary contribution. Replacing these two effects by a single effective fractional order generally suppresses the competition between the corresponding memory kernels. The coefficient $a$ measures the relative strength of the lower-order memory term, and changes in $a$ may alter not only the size but also the topology of the stability region. Consequently, a mixed-order model is a natural minimal framework for determining how multiple memory scales influence stability and bifurcation behavior in discrete systems.

\par The principal gap is that the cited theories do not provide a unified stability analysis for a forward-difference equation containing one order above one and a second order at or below one. Our earlier model \cite{chevala2025stability}, in particular, was restricted to $0 < \beta < \alpha \leq 1$ and emphasized bifurcations generated by a negative lower-order coefficient. The present work addresses the genuinely higher-order regime $0 < \beta \leq 1 < \alpha \leq 2$ with $a > 0$. It treats real and complex spectral parameters in a common framework, establishes a winding-number stability criterion, transfers that criterion to a nonlinear map, and studies the arbitrary-order family $N - 1 < \alpha \leq N$. These features are not available simultaneously in the existing literature.

\par The mathematical contributions of this work are fourfold. First, for the two-term linear equation, we derive the characteristic equation by the Z-transform and characterize asymptotic stability through the winding number of its unit-circle image. Second, we prove the formulas $a_1 = 2^{\alpha-\beta}$ and $a_2 = 2^{\alpha-\beta}(4 - \alpha)/(2 - \beta)$, which correspond to an endpoint collision and a stationary point of the stability boundary. Third, we transfer the linear criterion to a nonlinear higher-order logistic map through linearization and examine the distinct behavior of its trivial and nontrivial equilibria. Fourth, for the one-term family $\Delta^\alpha x(t) = (c - 1)x(t + \alpha - N)$, we derive the boundary for arbitrary $N$ and prove that no stability region exists for $N \geq 3$. These results provide a unified description of the effects of mixed fractional orders, higher-order dynamics, and real or complex parameters on discrete-time stability.

\par This paper consists of two parts. In the first part (Section \ref{1model}), we analyze the stability and bifurcations of a two-term linear fractional difference equation with orders $\alpha \in (1,2]$ and $\beta \in (0,1]$ for real and complex parameter values. We also consider a nonlinear logistic map and investigate the stability of its trivial and nontrivial equilibrium points through numerical experiments. In the second part (Section \ref{2model}), we study the stability of higher-order, one-term linear fractional difference equations with $N - 1 < \alpha \leq N$, where $N \in \mathbb{N}$. Conclusions are presented in Section \ref{con.}.

\section{Preliminaries}
In this section, we recall the definitions and results needed in the subsequent analysis \cite{mozyrska2015z,elaydi2005introduction}. The base point of each fractional operator is set to zero, and the step size is $h = 1$.
\begin{deff} \cite{mozyrska2015z}
 For a function $x:\mathbb{N}_0 \to \mathbb{R}$, the fractional sum of order $\beta > 0$ is given by
	\begin{equation*} 
		(\Delta^{-\beta} x)(t)
		=\frac{1}{\Gamma(\beta)} 
		\sum_{s=0}^{n} \frac{\Gamma(n-s+\beta)}{\Gamma(n-s+1)}\, x(s),
	\end{equation*}
	where $t = \beta + n$, $n \in \mathbb{N}_0$, and $\beta \in \mathbb{R}^{+}$.
\end{deff}
\begin{deff} \cite{mozyrska2015z} 
	Let $\mu > 0$ and $m - 1 < \mu < m$, where $m \in \mathbb{N}$ and $m = \lceil\mu\rceil$.
	The $\mu$th fractional Caputo-like difference is defined as
	\begin{equation*}
		\Delta^\mu x(t) = \Delta^{-(m-\mu)}\bigl(\Delta^m x(t)\bigr),
	\end{equation*}
	where $t \in \mathbb{N}_{m-\mu}$ and
	\begin{equation*}
		\Delta^m x(t) = \sum_{k=0}^m \binom{m}{k}(-1)^{m-k}x(t+k).
	\end{equation*}
\end{deff}
\begin{deff}\cite{elaydi2005introduction,mozyrska2015z} 
	The Z-transform of a sequence $\{y(n)\}_{n=0}^\infty$ is the complex-valued function
	\[
	Y(z) = \mathcal{Z}[y](z) = \sum_{k=0}^\infty y(k)z^{-k},
	\]
	where $z$ is a complex number for which the series converges absolutely.
\end{deff}
\begin{prop}\cite{elaydi2005introduction,mozyrska2015z}
	The convolution $\phi \ast x$ of functions $\phi$ and $x$ defined on $\mathbb{N}_0$ is
	\[
	(\phi \ast x)(n) = \sum_{s=0}^n \phi(n-s)x(s) = \sum_{s=0}^n \phi(s)x(n-s).
	\]
	Its Z-transform satisfies $\mathcal{Z}[\phi \ast x](z) = \mathcal{Z}[\phi](z)\,\mathcal{Z}[x](z)$.
\end{prop}
\begin{deff}\cite{rudin1987real}
	Let $\eta:[0,2\pi]\to\mathbb{C}$ be a closed, piecewise continuously differentiable curve and let $\lambda\notin\eta([0,2\pi])$. The winding number of $\eta$ about $\lambda$ is
	\begin{equation*}
		\operatorname{Ind}_{\eta}(\lambda)
		= \frac{1}{2\pi\mathrm{i}}\int_{0}^{2\pi}
		\frac{\eta'(\theta)}{\eta(\theta)-\lambda}\,d\theta.
	\end{equation*}
	It is constant on every connected component of $\mathbb{C}\setminus\eta([0,2\pi])$.
\end{deff}
\begin{prop}[Argument principle]\label{prop:argument-principle}\cite{rudin1987real}
	Let $\Omega\subset\mathbb{C}$ be a bounded domain whose positively oriented boundary $\partial\Omega$ is a finite union of piecewise continuously differentiable simple closed curves. Let $F$ be meromorphic on an open set containing $\overline{\Omega}$, and assume that $F$ has neither zeros nor poles on $\partial\Omega$. If $N_F$ and $P_F$ denote, respectively, the numbers of zeros and poles of $F$ in $\Omega$, counted with multiplicity, then
	\begin{equation}
		\frac{1}{2\pi\mathrm{i}}\int_{\partial\Omega}
		\frac{F'(z)}{F(z)}\,dz=N_F-P_F.
		\label{argument-principle}
	\end{equation}
	Equivalently, the left-hand side is the winding number of the image curve $F(\partial\Omega)$ about the origin. More generally, for any $\lambda\notin F(\partial\Omega)$,
	\begin{equation*}
		\operatorname{Ind}_{F(\partial\Omega)}(\lambda)
		= N_{F-\lambda}-P_{F-\lambda},
	\end{equation*}
	where the zeros and poles are counted in $\Omega$. In particular, if $F$ is analytic in $\Omega$, then $P_F=0$, and the integral in \eqref{argument-principle} counts exactly the zeros of $F$ in $\Omega$.

	For an annulus $\Omega=\{z:r<|z|<R\}$, positive orientation means that the outer circle $|z|=R$ is traversed counterclockwise and the inner circle $|z|=r$ is traversed clockwise. Consequently,
	\begin{equation*}
		N_F-P_F
		=\frac{1}{2\pi\mathrm{i}}\int_{|z|=R}\frac{F'(z)}{F(z)}\,dz
		-\frac{1}{2\pi\mathrm{i}}\int_{|z|=r}\frac{F'(z)}{F(z)}\,dz,
	\end{equation*}
	when both circular integrals on the right are parametrized counterclockwise. This annular form is the one used below to count characteristic zeros outside the unit disk.
\end{prop}
\section{The model} \label{1model}
Consider the initial-value problem (IVP) for $0 < \beta \leq 1 < \alpha \leq 2$:
\begin{equation} 
\Delta^{\alpha}x(t) + a\,\Delta^{\beta}x(t + \alpha - \beta - 1)
= f\bigl(x(t + \alpha - 2)\bigr) - x(t + \alpha - 2), \label{1}
\end{equation}
\begin{equation}
   x(0) = x_0, \qquad x(1) = x_1, \label{2}
\end{equation}
where $f \in C^1$, $t \in \mathbb{N}_{2-\alpha}$, and $a > 0$.
At the integer endpoints $\alpha=2$ or $\beta=1$, the corresponding operator is understood as the ordinary forward difference. For $\nu>0$, set
\begin{equation*}
\phi_\nu(m)=\frac{\Gamma(m+\nu)}{\Gamma(\nu)\Gamma(m+1)},
\qquad m\in\mathbb{N}_0,
\end{equation*}
and use the limiting convention $\phi_0(0)=1$ and $\phi_0(m)=0$ for $m\geq1$. The definitions of the fractional sum and the Caputo-like fractional difference give the exact recurrence
\begin{align}
x(n)={}&2x(n-1)-2x(n-2)+f\bigl(x(n-2)\bigr) \nonumber\\
&-\sum_{s=0}^{n-3}\phi_{2-\alpha}(n-2-s)\Delta^2x(s) \nonumber\\
&-a\sum_{s=0}^{n-2}\phi_{1-\beta}(n-2-s)\Delta x(s),
\qquad n\geq2, \label{seqrep}
\end{align}
where the first sum is empty when $n=2$. In particular,
\begin{align}
x(2)={}&(2-a)x(1)+(a-2)x(0)+f\bigl(x(0)\bigr), \\
x(3)={}&(\alpha-a)x(2)+(2-2\alpha+a\beta)x(1) \nonumber\\
&+(\alpha-2+a-a\beta)x(0)+f\bigl(x(1)\bigr).
\end{align}
Equation \eqref{seqrep} is the recurrence used for the numerical simulations in this revised version.
\par We now consider the linear function $f(x) = bx$, where $b \in \mathbb{C}$, and derive the corresponding characteristic equation.
 \subsection{Stability analysis}
 Applying the Z-transform to both sides of \eqref{1} with $f(x) = bx$ gives
 \begin{align*}
 &\left(\frac{z}{z - 1}\right)^{2-\alpha}
 \Bigl[(z^2 - 2z + 1)X(z) - (z^2 - 2z)x(0) - zx(1)\Bigr] \\
 &\quad + a\left(\frac{z}{z - 1}\right)^{1-\beta}
 \Bigl[(z - 1)X(z) - zx(0)\Bigr] = (b - 1)X(z),
 \end{align*}
 where $X$ is the Z-transform of $x$. Rearranging yields
 \begin{align}
 &\left[z^2\left(\frac{z}{z - 1}\right)^{-\alpha}
 + az\left(\frac{z}{z - 1}\right)^{-\beta} - b + 1\right]X(z) \nonumber\\
 &\quad = \left[z(z - 2)\left(\frac{z}{z - 1}\right)^{2-\alpha}
 + az\left(\frac{z}{z - 1}\right)^{1-\beta}\right]x(0)
 + z\left(\frac{z}{z - 1}\right)^{2-\alpha}x(1). \label{ztran}
 \end{align}
 
Setting the coefficient of $X(z)$ in \eqref{ztran} equal to zero yields the characteristic equation associated with \eqref{1}:
 \begin{equation}
     z^2(1 - z^{-1})^\alpha + az(1 - z^{-1})^\beta + 1 = b. \label{chareq}
 \end{equation}
 Stability requires $|z| < 1$. The boundary of the stability region is therefore obtained by setting $z = e^{\mathrm{i}\theta}$ in \eqref{chareq}, which gives
 \begin{equation}
     e^{2\mathrm{i}\theta}(1 - e^{-\mathrm{i}\theta})^\alpha
     + ae^{\mathrm{i}\theta}(1 - e^{-\mathrm{i}\theta})^\beta + 1 = b, \label{bval}
 \end{equation}
 where $\theta \in [0,2\pi]$. Simplifying \eqref{bval}, we obtain
 \begin{align}
 b ={}& 2^\alpha\left(\sin\frac{\theta}{2}\right)^\alpha
 e^{\mathrm{i}\left[\frac{\alpha\pi}{2}+\theta\left(2-\frac{\alpha}{2}\right)\right]} \nonumber\\
 &+ a\,2^\beta\left(\sin\frac{\theta}{2}\right)^\beta
 e^{\mathrm{i}\left[\frac{\beta\pi}{2}+\theta\left(1-\frac{\beta}{2}\right)\right]} + 1. \label{simbval}
\end{align}
This expression describes the boundary of the stability region in the complex plane. Separating the real and imaginary parts of \eqref{simbval} gives the parametric representation
 \begin{equation}
     \gamma(\theta) = \bigl(\gamma_1(\theta),\gamma_2(\theta)\bigr), \label{boundarycurve}
 \end{equation}
 where 
\begin{align}
\gamma_1(\theta) ={}& 2^\alpha\left(\sin\frac{\theta}{2}\right)^\alpha
\cos\left(\frac{\alpha\pi}{2}+\theta\left(2-\frac{\alpha}{2}\right)\right) \nonumber\\
&+ a\,2^\beta\left(\sin\frac{\theta}{2}\right)^\beta
\cos\left(\frac{\beta\pi}{2}+\theta\left(1-\frac{\beta}{2}\right)\right) + 1, \label{realpart}
\end{align}
  and
\begin{align}
\gamma_2(\theta) ={}& 2^\alpha\left(\sin\frac{\theta}{2}\right)^\alpha
\sin\left(\frac{\alpha\pi}{2}+\theta\left(2-\frac{\alpha}{2}\right)\right) \nonumber\\
&+ a\,2^\beta\left(\sin\frac{\theta}{2}\right)^\beta
\sin\left(\frac{\beta\pi}{2}+\theta\left(1-\frac{\beta}{2}\right)\right). \label{imgpart}
\end{align}
Here, $\gamma_1(\theta)$ and $\gamma_2(\theta)$ are the real and imaginary parts, respectively, of $b$ in \eqref{simbval}.

\begin{thm}[Characteristic-root and winding-number criterion]\label{thm:mixed-stability}
Let $0 < \beta \leq 1 < \alpha \leq 2$ and $a > 0$, and define
\begin{equation*}
D_b(z)=z^2(1-z^{-1})^\alpha+az(1-z^{-1})^\beta+1-b,
\end{equation*}
where the principal branches of the fractional powers are used. If $b\notin\gamma([0,2\pi])$, then the zero solution of the linear equation associated with \eqref{1} is asymptotically stable if and only if
\begin{equation}\label{winding-stability}
\operatorname{Ind}_{\gamma}(b)=2.
\end{equation}
Equivalently, the stability region is
\begin{equation*}
\mathcal{S}_{\alpha,\beta,a}
=\left\{b\in\mathbb{C}\setminus\gamma([0,2\pi]):
\operatorname{Ind}_{\gamma}(b)=2\right\}.
\end{equation*}
Its boundary is the curve \eqref{simbval}, or equivalently \eqref{realpart}--\eqref{imgpart}.
\end{thm}

\begin{figure}[H]
\centering
\includegraphics[width=0.95\textwidth]{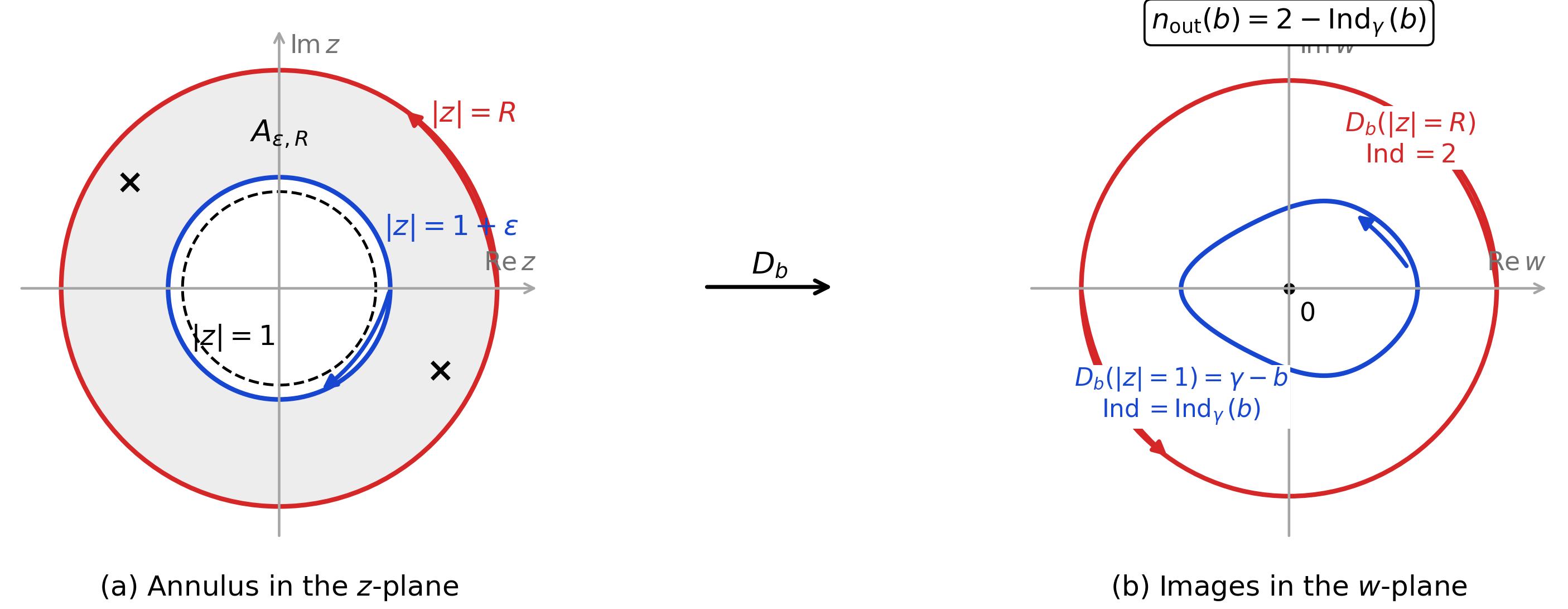}
\caption{Schematic illustration of the argument-principle setup in Theorem \ref{thm:mixed-stability}. The annulus $A_{\varepsilon,R}=\{z:1+\varepsilon<|z|<R\}$ counts characteristic zeros outside the unit disk; the crosses indicate possible zeros counted by $n_{\mathrm{out}}(b)$. Its outer boundary is counterclockwise, whereas its inner boundary has the clockwise orientation induced by the annulus. Under $D_b$, the outer-boundary image has winding number two for large $R$, while the inner-boundary image tends to $\gamma-b$ as $\varepsilon\downarrow0$. Hence $n_{\mathrm{out}}(b)=2-\operatorname{Ind}_{\gamma}(b)$, and stability corresponds to $\operatorname{Ind}_{\gamma}(b)=2$. The image curves are schematic and not drawn to scale.}
\label{fig:mixed-winding-setup}
\end{figure}

\begin{proof}
The contour and its boundary images are summarized in Figure \ref{fig:mixed-winding-setup}.
Equation \eqref{ztran} has the form
\begin{equation*}
D_b(z)X(z)=P(z),
\end{equation*}
where $P$ depends only on the initial conditions. By the characteristic-root criterion for delta fractional difference equations, the zero solution is asymptotically stable precisely when every characteristic zero lies in the open unit disk. A zero with $|z|>1$ produces a growing mode, while a zero on $|z|=1$ prevents asymptotic decay \cite{stanislawski2013stability,abu2013asymptotic,vcermak2015explicit,wei2024explicit}.

The role of the region $|z|>1$ in the proof is only to count the unstable characteristic zeros; it is not an assumption that such a zero exists. This exterior region is used because the principal fractional powers are analytic there. Indeed, if $|z|>1$, then
\begin{equation*}
\operatorname{Re}(1-z^{-1})
\geq 1-|z|^{-1}>0.
\end{equation*}
Thus, $1-z^{-1}$ does not meet the branch cut $(-\infty,0]$, and $D_b$ is analytic for $|z|>1$.

Let $n_{\mathrm{out}}(b)$ denote the number of zeros of $D_b$, counted with multiplicity, that satisfy $|z|>1$. Choose $R>1$ sufficiently large that all such zeros lie in $|z|<R$. For $\varepsilon>0$, apply the argument principle to the annulus
\begin{equation*}
1+\varepsilon<|z|<R.
\end{equation*}
Both circular integrals below are written counterclockwise; therefore, the inner-circle contribution is subtracted. Letting $\varepsilon\downarrow0$ gives
\begin{equation}\label{nout-mixed}
n_{\mathrm{out}}(b)
=\frac{1}{2\pi\mathrm{i}}\int_{|z|=R}\frac{D_b'(z)}{D_b(z)}\,dz
-\frac{1}{2\pi\mathrm{i}}\int_{|z|=1}\frac{D_b'(z)}{D_b(z)}\,dz.
\end{equation}
The limit at the inner circle is valid because $b\notin\gamma([0,2\pi])$; in particular, $b\neq1$ and $D_b$ has no zero on $|z|=1$.

We evaluate the two terms in \eqref{nout-mixed}. First, on $z=e^{\mathrm{i}\theta}$,
\begin{equation*}
D_b(e^{\mathrm{i}\theta})=\gamma(\theta)-b.
\end{equation*}
Using $dz=\mathrm{i}e^{\mathrm{i}\theta}d\theta$ and differentiating the preceding identity with respect to $\theta$, we obtain
\begin{align*}
\frac{1}{2\pi\mathrm{i}}\int_{|z|=1}\frac{D_b'(z)}{D_b(z)}\,dz
&=\frac{1}{2\pi\mathrm{i}}\int_0^{2\pi}
\frac{\gamma'(\theta)}{\gamma(\theta)-b}\,d\theta\\
&=\operatorname{Ind}_{\gamma}(b).
\end{align*}
Second, as $|z|\to\infty$,
\begin{equation*}
D_b(z)=z^2+O(z),
\qquad
\frac{D_b(z)}{z^2}\longrightarrow1
\end{equation*}
uniformly on $|z|=R$. To justify the winding number of the outer image
explicitly, write $z=Re^{\mathrm{i}\theta}$. Then
\begin{equation*}
D_b\bigl(Re^{\mathrm{i}\theta}\bigr)
=R^2e^{2\mathrm{i}\theta}\bigl(1+O(R^{-1})\bigr),
\qquad 0\leq\theta\leq2\pi,
\end{equation*}
where the error estimate is uniform in $\theta$. The leading curve
$R^2e^{2\mathrm{i}\theta}$ makes two counterclockwise turns around the
origin because its argument increases from $0$ to $4\pi$. Choose $R$ so
large that
\begin{equation*}
\sup_{|z|=R}\left|\frac{D_b(z)}{z^2}-1\right|<1.
\end{equation*}
Then the curves
\begin{equation*}
H_s(z)=z^2\left[1+s\left(\frac{D_b(z)}{z^2}-1\right)\right],
\qquad 0\leq s\leq1,
\end{equation*}
do not pass through the origin. Thus, $D_b(|z|=R)$ is homotopic in
$\mathbb{C}\setminus\{0\}$ to the twice-traversed circle $z^2(|z|=R)$,
so both curves have winding number two. By the definition of winding
number, this gives
\begin{equation*}
\frac{1}{2\pi\mathrm{i}}\int_{|z|=R}\frac{D_b'(z)}{D_b(z)}\,dz=2.
\end{equation*}
Substitution into \eqref{nout-mixed} yields the root-count identity
\begin{equation}\label{root-count-mixed}
n_{\mathrm{out}}(b)=2-\operatorname{Ind}_{\gamma}(b).
\end{equation}

We can now state the equivalence explicitly. Since $b\notin\gamma([0,2\pi])$, no characteristic zero lies on $|z|=1$. Therefore,
\begin{align*}
\text{asymptotic stability}
&\iff \text{there is no characteristic zero with }|z|>1\\
&\iff n_{\mathrm{out}}(b)=0\\
&\overset{\eqref{root-count-mixed}}{\iff}
2-\operatorname{Ind}_{\gamma}(b)=0\\
&\iff \operatorname{Ind}_{\gamma}(b)=2.
\end{align*}
Conversely, if the system is unstable and $b$ is not on the boundary, then at least one characteristic zero satisfies $|z|>1$. Hence $n_{\mathrm{out}}(b)\geq1$, and \eqref{root-count-mixed} shows that $\operatorname{Ind}_{\gamma}(b)\neq2$.

Finally, substituting $z=e^{\mathrm{i}\theta}$ into $D_b(z)=0$ gives \eqref{simbval}; separating its real and imaginary parts gives \eqref{realpart}--\eqref{imgpart}.
\end{proof}
\subsection[Bifurcations for complex parameter b]{Bifurcations for $0 < \beta \leq 1 < \alpha \leq 2$, $a > 0$, and $b \in \mathbb{C}$} \label{imaginary case}
The two distinguished parameter values observed in the boundary geometry can be derived analytically.

\begin{thm}[Boundary bifurcation values]\label{thm:bifurcation-values}
Let $0 < \beta \leq 1 < \alpha \leq 2$ and let $\gamma$ be given by \eqref{simbval}. Define
\begin{equation}\label{a1a2}
a_1=2^{\alpha-\beta},
\qquad
a_2=2^{\alpha-\beta}\frac{4-\alpha}{2-\beta}.
\end{equation}
Then:
\begin{enumerate}
\item $\gamma(0)=\gamma(\pi)$ if and only if $a=a_1$;
\item $\gamma'(\pi)=0$ if and only if $a=a_2$;
\item $0<a_1<a_2$.
\end{enumerate}
Thus, $a_1$ is the unique value for which $\gamma(0)=\gamma(\pi)$, and $a_2$ is the unique value for which the boundary is stationary at $\theta=\pi$.
\end{thm}

\begin{proof}
Since $\sin(0)=0$, equation \eqref{simbval} gives $\gamma(0)=1$. At $\theta=\pi$, the two phase factors satisfy
\begin{equation*}
e^{\mathrm{i}[\alpha\pi/2+\pi(2-\alpha/2)]}=1,
\qquad
e^{\mathrm{i}[\beta\pi/2+\pi(1-\beta/2)]}=-1.
\end{equation*}
Therefore,
\begin{equation*}
\gamma(\pi)=1+2^\alpha-a\,2^\beta.
\end{equation*}
The equality $\gamma(0)=\gamma(\pi)$ is equivalent to $2^\alpha=a\,2^\beta$, and hence to $a=a_1$.

To obtain the second value, differentiate \eqref{simbval}. The derivatives of the radial factors vanish at $\theta=\pi$ because $\cos(\pi/2)=0$. Consequently,
\begin{align*}
\gamma'(\pi)
&=\mathrm{i}\,2^\alpha\left(2-\frac{\alpha}{2}\right)
-\mathrm{i}\,a\,2^\beta\left(1-\frac{\beta}{2}\right)\\
&=\mathrm{i}\left[2^{\alpha-1}(4-\alpha)
-a\,2^{\beta-1}(2-\beta)\right].
\end{align*}
Thus, $\gamma'(\pi)=0$ if and only if
\begin{equation*}
a=2^{\alpha-\beta}\frac{4-\alpha}{2-\beta}=a_2.
\end{equation*}
Finally,
\begin{equation*}
\frac{a_2}{a_1}=\frac{4-\alpha}{2-\beta}>1,
\end{equation*}
because $2-\alpha+\beta>0$. Hence $0<a_1<a_2$.
\end{proof}

For numerical classification, let $\theta_j=2\pi j/M$ and define the discrete winding-number approximation
\begin{equation}\label{discrete-winding}
w_a(b)=\frac{1}{2\pi}\sum_{j=0}^{M-1}
\operatorname{Arg}\left(
\frac{\gamma(\theta_{j+1})-b}{\gamma(\theta_j)-b}
\right),
\qquad \theta_M=2\pi,
\end{equation}
where $\operatorname{Arg}\in(-\pi,\pi]$. This is the standard signed-angle summation applied to the polygonal approximation of the curve \cite{hormann2001point}. For sufficiently large $M$, the value of \eqref{discrete-winding} is an integer and equals $\operatorname{Ind}_{\gamma}(b)$. In the computations below, we used $M=10^5$ and verified that doubling $M$ did not change the resulting integer. Thus, Theorem \ref{thm:mixed-stability} classifies a point as stable precisely when $w_a(b)=2$.

For the parameter values displayed in Fig. \ref{case1all}, the index-two component undergoes the endpoint collision at $a_1$ and collapses at the stationary-point value $a_2$. The bounded unstable components have winding number one, whereas the unbounded exterior has winding number zero.

\begin{figure}[H]
	\centering
	\subfigure[$a<a_1$]{
		\includegraphics[height=3.5in,width=3.5in,keepaspectratio]{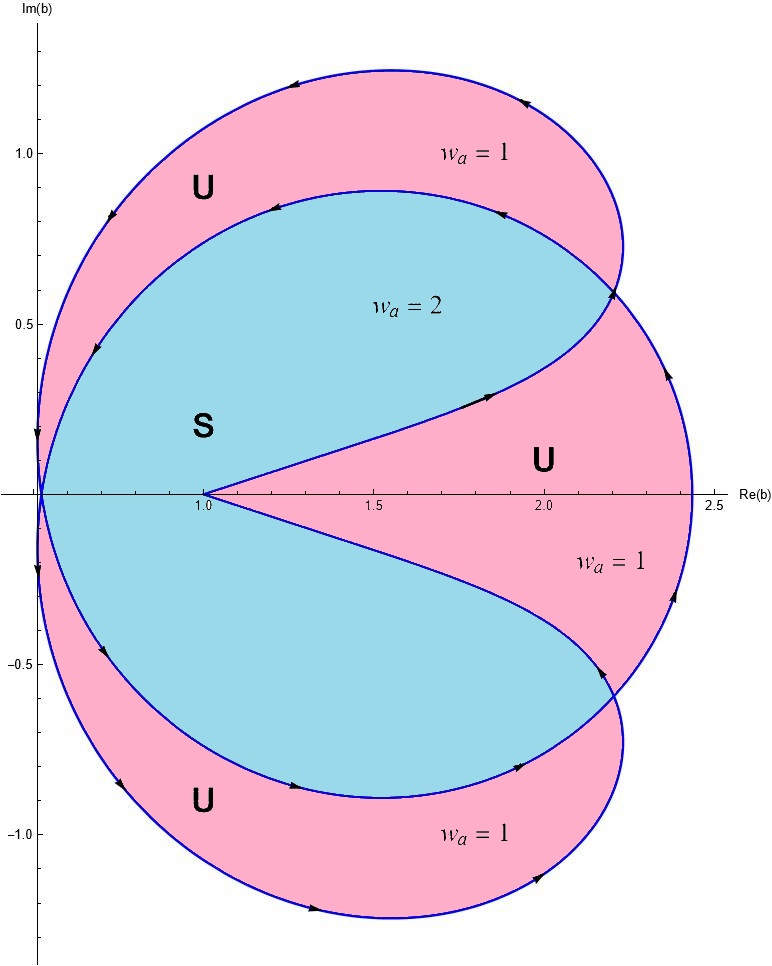}
		\label{case1a}
	} \hspace{0.3cm}
	\subfigure[$a_1$]{
\includegraphics[height=3.5in,width=3.5in,keepaspectratio]{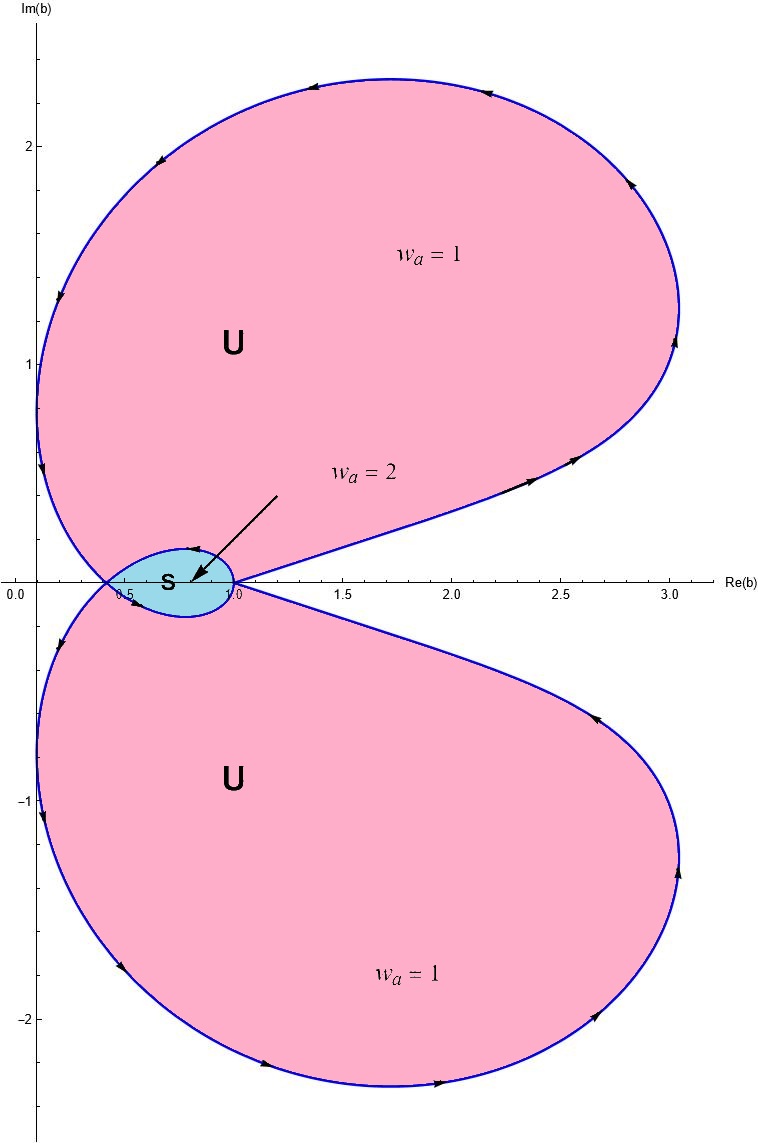}
		\label{case1b}
	} \hspace{0.3cm}
	\subfigure[$a_1<a<a_2$]{
		\includegraphics[height=3.5in,width=3.5in,keepaspectratio]{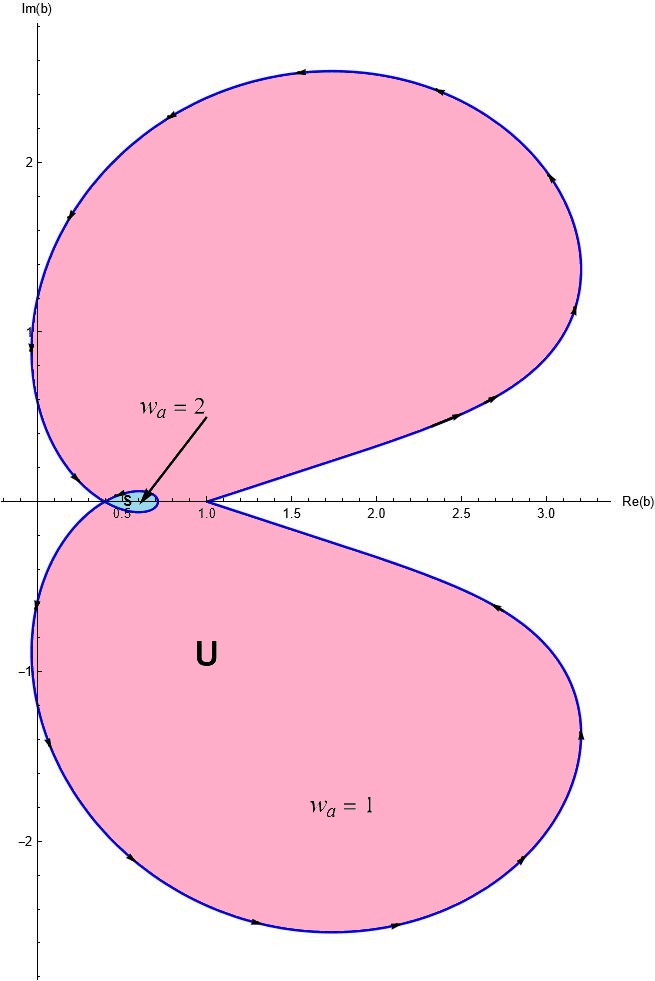}
		\label{case1c}
	} \hspace{0.3cm}
	\subfigure[$a_2$]{
		\includegraphics[height=3.5in,width=3.5in,keepaspectratio]{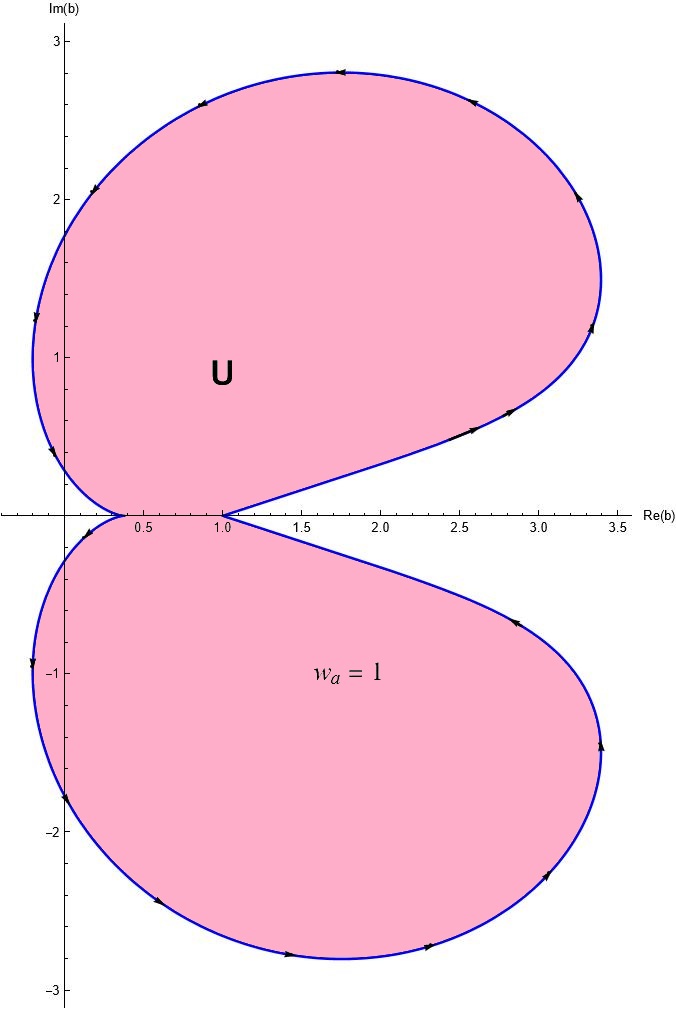}
		\label{case1d}
	}
	\caption{Stability regions for $\alpha = 1.9$ and $\beta = 0.2$ at $a = 2$ (a), $a = 3.24901$ (b), $a = 3.5$ (c), and $a = 3.79051$ (d). The displayed values are $w_a(b)=\operatorname{Ind}_{\gamma}(b)$. Hence, the regions with $w_a=2$ are stable, the bounded regions with $w_a=1$ are unstable, and the unbounded exterior, for which $w_a=0$, is unstable.}
	\label{case1all}
\end{figure}
\subsubsection{Numerical experiments}
In Table \ref{tab1}, we set $\alpha = 1.9$, $\beta = 0.2$, $x(0) = 0.1$, and $x(1) = 0.2$. The corresponding bifurcation values are $a_1 = 3.24901$ and $a_2 = 3.79051$. The winding number is evaluated using \eqref{discrete-winding}; the solution converges to zero exactly for the tested points with $w_a(b)=2$.
\begin{table}[H]
	\centering
	\renewcommand{\arraystretch}{1.5}
	\begin{tabular}{|c|c|c|c|}
		\hline
		\textbf{$a$} & \textbf{$b$} & \textbf{$w_a(b)$} & \textbf{Behavior} \\
		\hline
		
		$2$ & $1.891 - 0.624\mathrm{i}$ & $2$ & Stable \\
		\cline{2-4}
		& $1.661 + 0.9917\mathrm{i}$ & $1$ & Unstable \\
		\cline{2-4}
		& $2.168 - 0.7312\mathrm{i}$ & $1$ & Unstable \\
		\cline{2-4}
		& $2$ & $1$ & Unstable \\
		\hline
		
		$3.24901$ & $0.7$ & $2$ & Stable \\
		\cline{2-4}
		& $1.489 + 1.224\mathrm{i}$ & $1$ & Unstable \\
		\cline{2-4}
		& $0.3925 - 0.1999\mathrm{i}$ & $1$ & Unstable \\
		\hline
		
		$3.5$ & $0.6$ & $2$ & Stable \\ \cline{2-4}
		& $0.9$ & $1$ & Unstable \\
		\hline
		
		$3.79051$ & $-0.08687 - 0.9862\mathrm{i}$ & $1$ & Unstable \\
		\hline
	\end{tabular}
	\caption{Winding numbers and solution behavior for $\alpha = 1.9$, $\beta = 0.2$, and different values of $a$ and $b$.} \label{tab1}
\end{table}

\begin{figure}[H]
	\centering
	\subfigure[$a = 2$, $b = 1.891 - 0.624\mathrm{i}$, $w_a=2$]{
		\includegraphics[height=2.5in,width=2.5in,keepaspectratio]{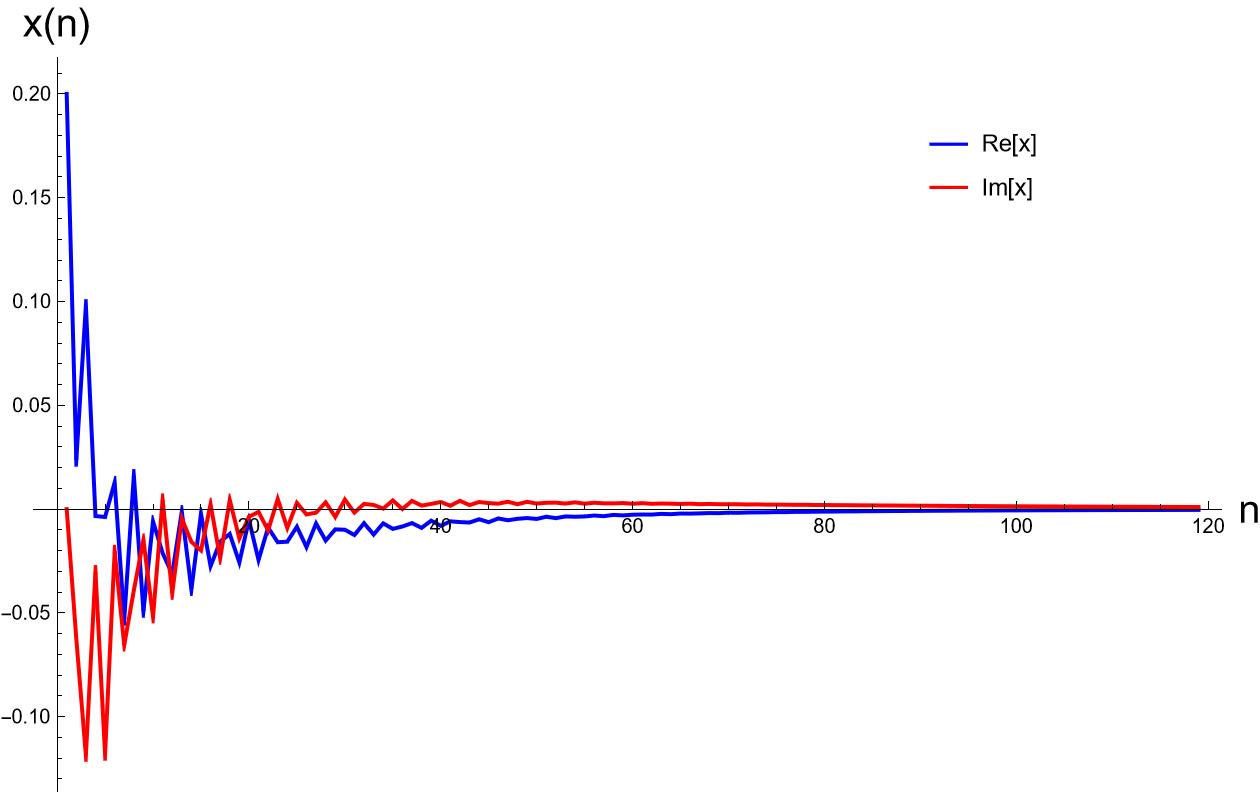}
		\label{coex1a}
	} \hspace{0.3cm}
	\subfigure[$a = 2$, $b = 1.661 + 0.9917\mathrm{i}$, $w_a=1$]{
		\includegraphics[height=2.5in,width=2.5in,keepaspectratio]{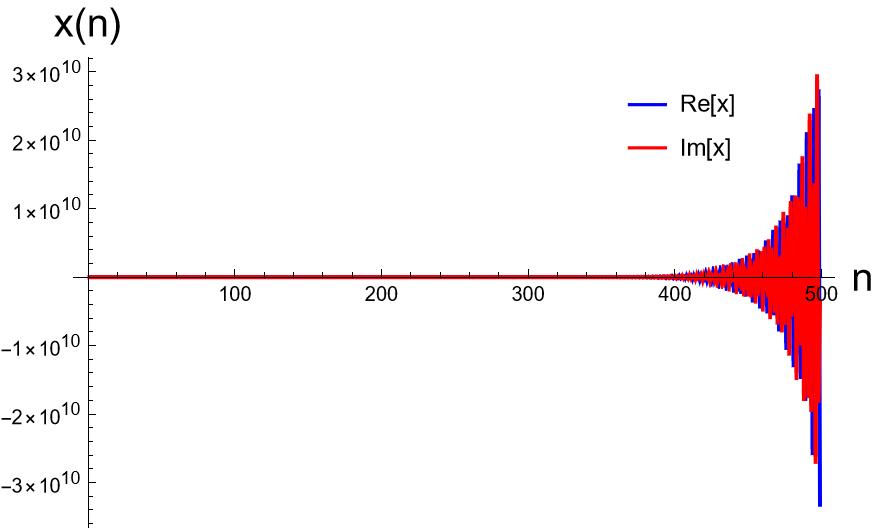}
		\label{coex1b}
	} \hspace{0.3cm}
	\subfigure[$a=3.5$, $b=0.6$, $w_a=2$]{
		\includegraphics[height=2.5in,width=2.5in,keepaspectratio]{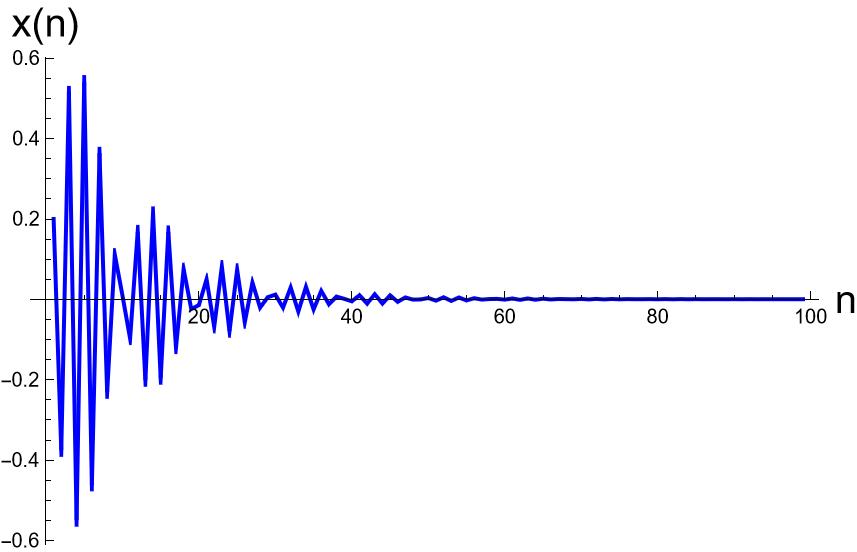}
		\label{coex1h}
	} \hspace{0.3cm}
	\subfigure[$a = 3.79051$, $b = -0.08687 - 0.9862\mathrm{i}$, $w_a=1$]{
		\includegraphics[height=2.5in,width=2.5in,keepaspectratio]{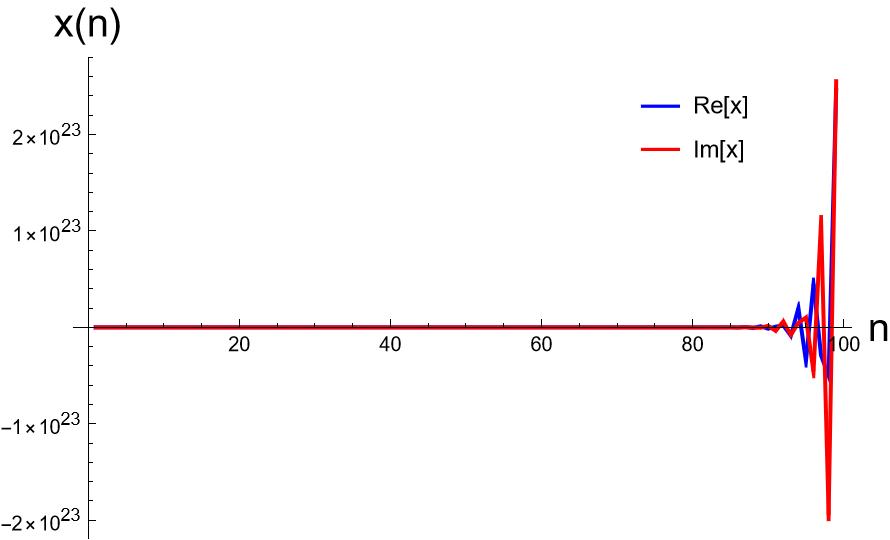}
		\label{coex1j}
	}
	\caption{Representative stable and unstable solutions of system \eqref{1} for $\alpha = 1.9$ and $\beta = 0.2$. The remaining parameter tests are summarized in Table \ref{tab1}.}
	\label{1c1}
\end{figure}

\subsection[Bifurcations for real parameter b]{Bifurcations for $0 < \beta \leq 1 < \alpha \leq 2$, $a > 0$, and $b \in \mathbb{R}$} \label{real case}
If $b \in \mathbb{R}$, the index-two component in the parameter ranges considered below intersects the real axis in an interval. To determine its left endpoint, we first fix $\alpha$, $\beta$, and $a$. We then solve $\gamma_2(\theta) = 0$ for $\theta \in (0,\pi)$ and substitute the resulting value of $\theta$ into \eqref{realpart}. The root was computed using the \texttt{FindRoot} command in Mathematica, and the interval was verified by the winding-number calculation \eqref{discrete-winding}.

\par The right endpoint is obtained by substituting $\theta = 0$ and $\theta = \pi$ into \eqref{realpart}. Thus,
\[
b=
\begin{cases}
	1, & 0 < a < a_{1},\\[6pt]
	1 + 2^{\alpha} - a\,2^{\beta}, & a_{1} \le a < a_2,
\end{cases}
\]
where $a_1$ and $a_2$ are the first and second bifurcation values, respectively.

\par Figure \ref{abplane} shows stability regions in the $(a,b)$-plane for several fixed pairs $(\alpha,\beta)$ in the cases $0 < a \leq a_1$ and $a_1 < a \leq a_2$. Numerical simulations verify the theoretical stability predictions for system \eqref{1}.

\begin{figure}[H]
    \centering
    \subfigure[$0<a \le a_1$ with $\alpha=1.8$ and $\beta=0.5$.]{
\includegraphics[height=2.5in,width=2.5in,keepaspectratio]{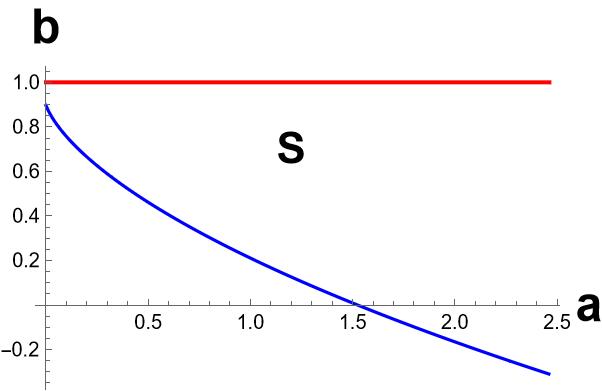}
        \label{ab1a}
     } \hspace{0.3cm}
     \subfigure[$a_1<a \le a_2$ with $\alpha=1.8$ and $\beta=0.5$.]{
     	\includegraphics[height=2.5in,width=2.5in,keepaspectratio]{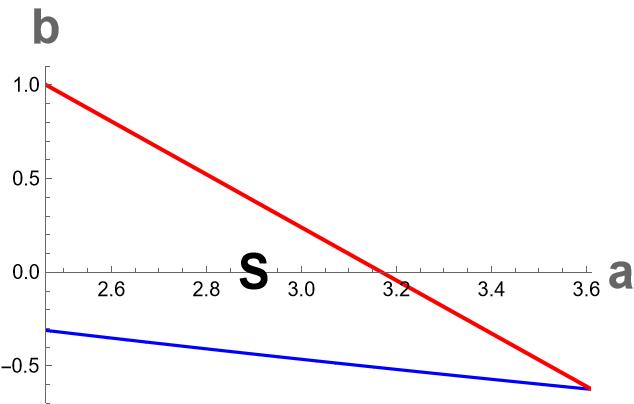}
     	\label{ab1b}
     } \hspace{0.3cm}
     \subfigure[$0<a \le a_1$ with $\alpha=1.2$ and $\beta=0.8$.]{
     	\includegraphics[height=2.5in,width=2.5in,keepaspectratio]{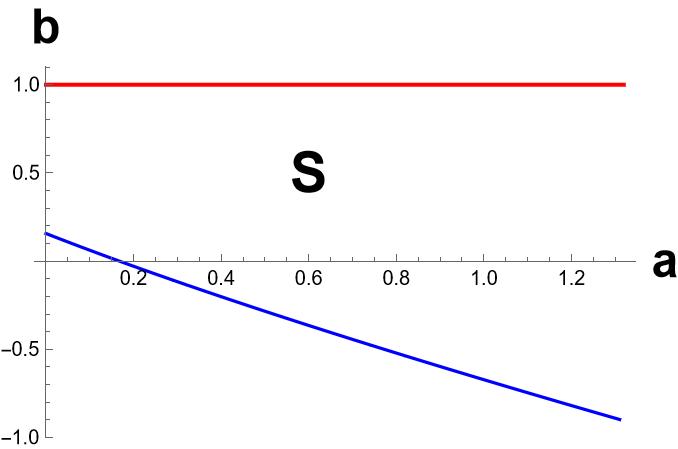}
     	\label{ab2a}
     } \hspace{0.3cm}
    \subfigure[$a_1<a\le a_2$ with $\alpha=1.2$ and $\beta=0.8$.]{
\includegraphics[height=2.5in,width=2.5in,keepaspectratio]{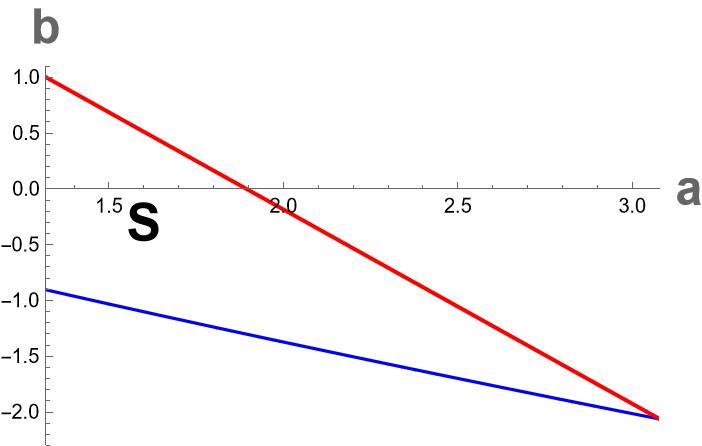}
        \label{ab2b}
    }
    \caption{Stability regions of system \eqref{1} in the $(a,b)$-plane for various values of $\alpha$ and $\beta$. The blue and red curves denote the left and right boundaries, respectively. Every point in the region marked S satisfies $\operatorname{Ind}_{\gamma}(b)=2$; points outside S have index different from two and are unstable.}
    \label{abplane}
\end{figure}

\begin{ex}
    Consider $\alpha = 1.8$ and $\beta = 0.5$. For $a = 1$, which is below the first bifurcation value $a_1 = 2.46229$, system \eqref{1} is stable for $b \in (0.210772,1)$. Formula \eqref{discrete-winding} gives $w_1(0.3)=2$ and $w_1(0.1)=0$. With initial conditions $x(0) = 0.1$ and $x(1) = 0.2$, the first solution converges to zero (Fig. \ref{Rex1a}), whereas the second is unbounded (Fig. \ref{Rex1c}).
    \par For $a = 3$, which lies between $a_1$ and $a_2 = 3.61136$, the stability interval of system \eqref{1} is
    \[
    b \in (-0.464274,0.239562).
    \]
    Here, $w_3(-0.3)=2$, and with $x(0) = 0.2$ and $x(1) = 0.3$, the solution converges to zero (Fig. \ref{Rex1e}).
    \par For $a = 4 > a_2$, the winding-number computation gives no index-two component on the real axis; in particular, $w_4(-3)=0$. With $x(0) = 0.3$ and $x(1) = 0.4$, the solution for $b = -3$ is unbounded (Fig. \ref{Rex1i}).
\end{ex}

\begin{figure}[H]
    \centering
    \subfigure[$b=0.3$, $w_1=2$]{
\includegraphics[height=2.5in,width=2.5in,keepaspectratio]{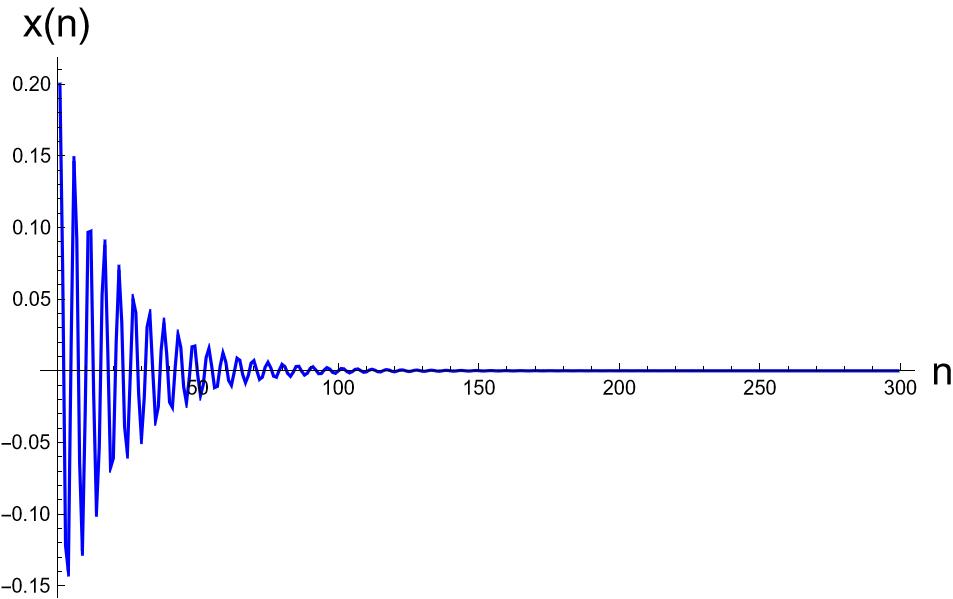}
        \label{Rex1a}
    } \hspace{0.3cm}
     \subfigure[$b=0.1$, $w_1=0$]{
\includegraphics[height=2.5in,width=2.5in,keepaspectratio]{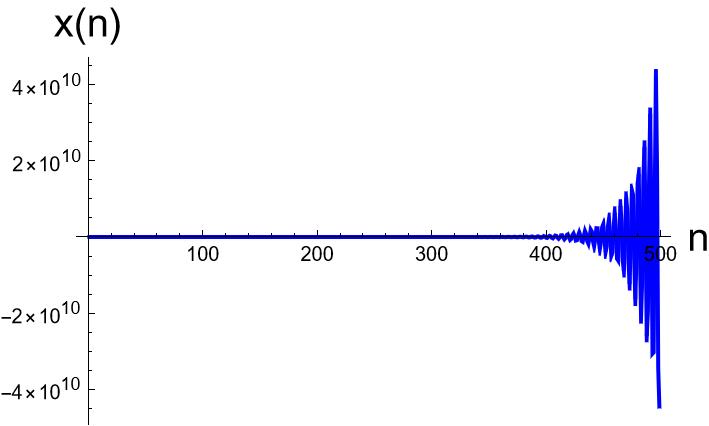}
        \label{Rex1c}
    } \hspace{0.3cm}
     \subfigure[$b=-0.3$, $w_3=2$]{
\includegraphics[height=2.5in,width=2.5in,keepaspectratio]{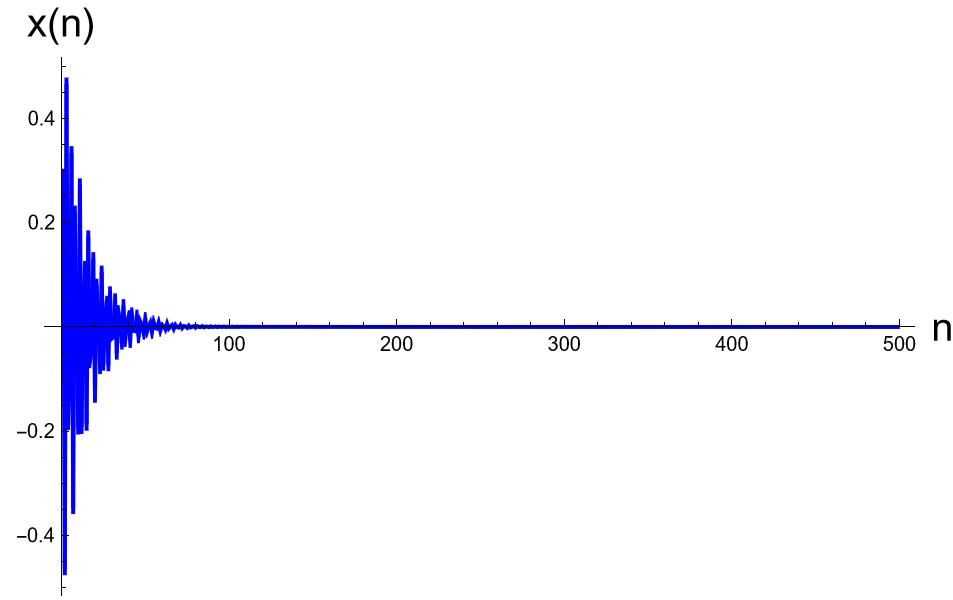}
        \label{Rex1e}
    } \hspace{0.3cm}
    \subfigure[$b=-3$, $w_4=0$]{
    \includegraphics[height=2.5in,width=2.5in,keepaspectratio]{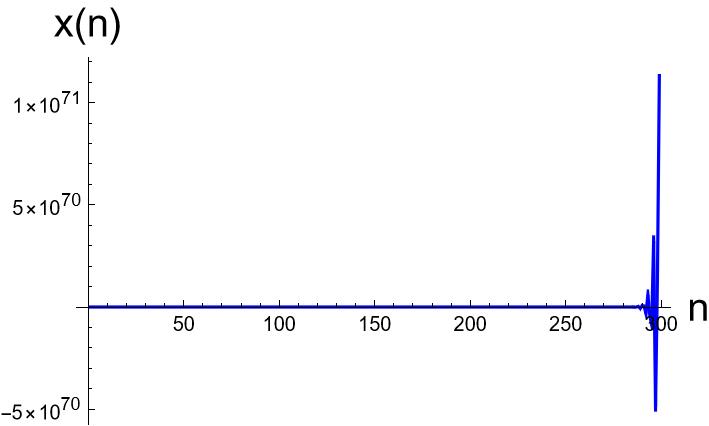}
    \label{Rex1i}
}
    \caption{Representative solutions of system \eqref{1} for values of $b$ inside and outside the stability interval.}
    \label{re}
\end{figure}
\subsection{Nonlinear map}
Consider the higher-order fractional difference equation \eqref{1} with $f(x) = \mu x(1 - x)$:
\begin{equation} 
	\Delta^{\alpha}x(t) + a\,\Delta^{\beta}x(t + \alpha - \beta - 1)
	= \mu x(t + \alpha - 2)\bigl(1 - x(t + \alpha - 2)\bigr) - x(t + \alpha - 2). \label{nonlinear}
\end{equation}
where $\mu$ is the control parameter. We refer to \eqref{nonlinear} as the higher-order logistic map. Its equilibrium points are $x_1^* = 0$ and, when $\mu\neq0$, $x_2^* = 1 - \frac{1}{\mu}$. Moreover, $b = f'(x_1^*) = \mu$ and $b = f'(x_2^*) = 2 - \mu$.

\begin{cor}[Linearized equilibrium criterion]\label{cor:nonlinear}
Let $x^*$ be an equilibrium of \eqref{1}, so that $f(x^*)=x^*$, and set $b^*=f'(x^*)$. The linearization about $x^*$ is asymptotically stable if and only if
\begin{equation*}
\operatorname{Ind}_{\gamma}(b^*)=2.
\end{equation*}
For the logistic nonlinearity, this condition becomes
\begin{equation*}
\operatorname{Ind}_{\gamma}(\mu)=2
\quad\text{at }x_1^*=0,
\qquad
\operatorname{Ind}_{\gamma}(2-\mu)=2
\quad\text{at }x_2^*=1-\frac{1}{\mu}.
\end{equation*}
\end{cor}

\begin{proof}
Write $x(t)=x^*+y(t)$. Since $f\in C^1$,
\begin{equation*}
f(x^*+y)-f(x^*)=f'(x^*)y+o(|y|).
\end{equation*}
Retaining the linear term gives equation \eqref{1} with $b=b^*$. Theorem \ref{thm:mixed-stability} therefore yields the stated linearized criterion. For the logistic map, direct differentiation gives $f'(x_1^*)=\mu$ and $f'(x_2^*)=2-\mu$.
\end{proof}

The numerical trajectories below use finite perturbations of the equilibria and illustrate the linearized prediction; they are not used to define an additional empirical stability threshold. Set $\alpha = 1.2$ and $\beta = 0.8$, for which $a_1 = 1.31951$ and $a_2 = 3.07885$. The trivial and nontrivial equilibria are examined in Examples \ref{NonlinearCaseEx1} and \ref{NonlinearCaseEx2}, respectively.

\begin{ex}	 
 We choose $x(0) = 0.1$ and $x(1) = 0.2$ in a neighborhood of $x_1^* = 0$. For $a = 0.5 < a_1$, the linearization at $x_1^*$ is asymptotically stable for $\mu \in (-0.284222,1)$. Since $f'(x_1^*)=\mu$, the corresponding winding numbers are $w_{0.5}(-0.27)=2$ and $w_{0.5}(-1)=0$. Accordingly, the solution for $\mu = -0.27$ converges to zero (Fig. \ref{nonlinear_ex1a}), whereas the solution for $\mu = -1$ does not converge to $x_1^*$ (Fig. \ref{nonlinear_ex1b}).
 
 \par For $a = 2.3$, which lies between $a_1$ and $a_2$, the theoretical stability interval is
 \[
 \mu \in (-1.57005,-0.707136).
 \]
 Here, $w_{2.3}(-1.4)=2$ and $w_{2.3}(-1.7)=0$. The first solution converges to zero (Fig. \ref{nonlinear_ex1c}), whereas the second moves away from $x_1^*$ (Fig. \ref{nonlinear_ex1d}).
 \par For $a = 3.5 > a_2$, the linearization at $x_1^*$ is unstable for every $\mu \in \mathbb{R}$. In particular, $w_{3.5}(-1.1)=1$ and $w_{3.5}(2.7)=0$, and both displayed solutions are unbounded (Figs. \ref{nonlinear_ex1e} and \ref{nonlinear_ex1f}).\label{NonlinearCaseEx1}
 \end{ex}

\begin{figure}[H]
 	\centering
 	\subfigure[$\mu=-0.27, a=0.5$]{
 		\includegraphics[height=2.5in,width=2.5in,keepaspectratio]{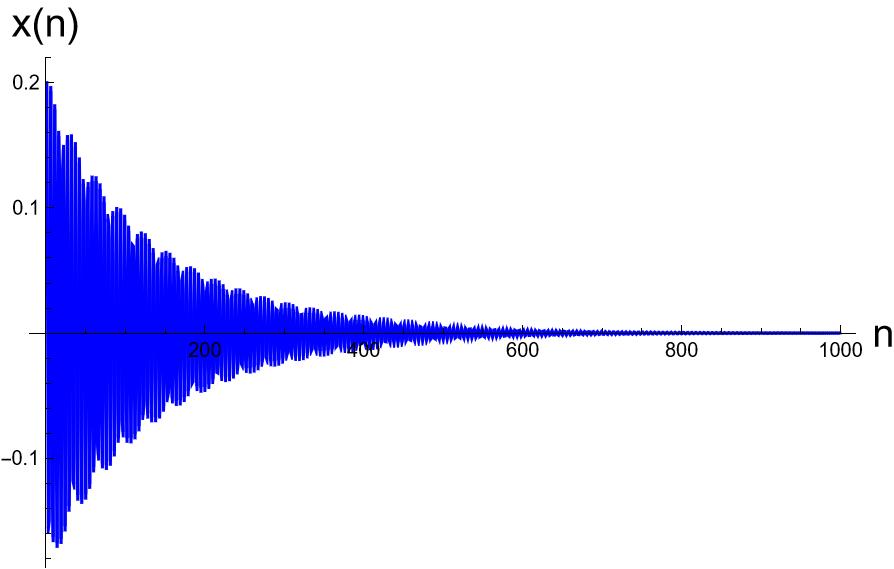}
 		\label{nonlinear_ex1a}
 	} \hspace{0.3cm}
 	\subfigure[$\mu=-1, a=0.5$]{
 		\includegraphics[height=2.5in,width=2.5in,keepaspectratio]{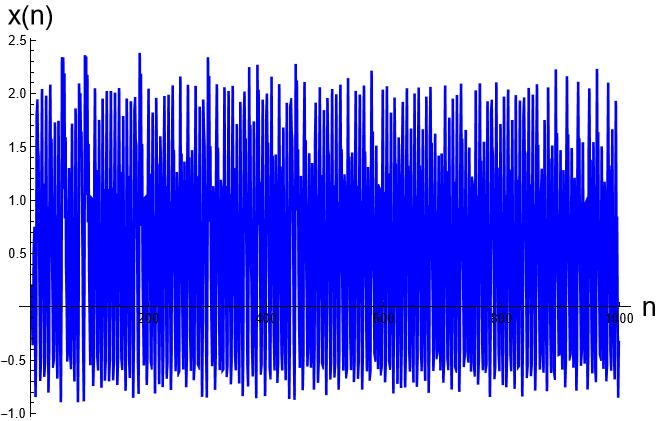}
 		\label{nonlinear_ex1b}
 	} \hspace{0.3cm}
 	\subfigure[$\mu=-1.4, a=2.3$]{
 		\includegraphics[height=2.5in,width=2.5in,keepaspectratio]{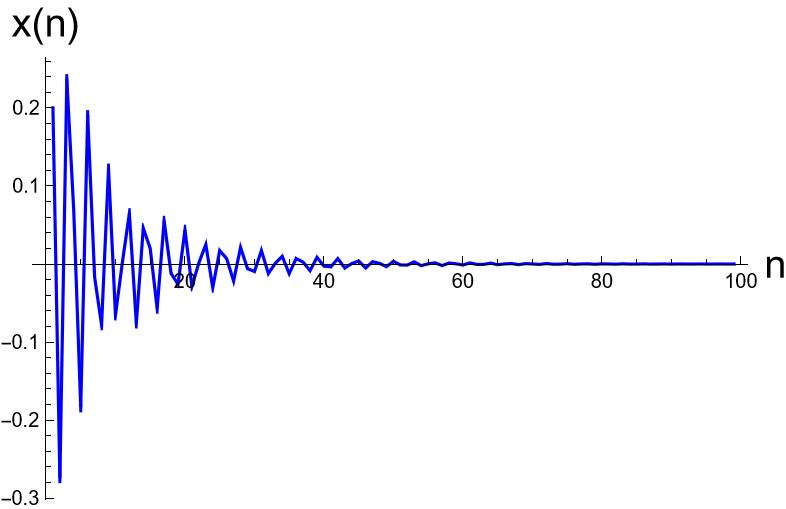}
 		\label{nonlinear_ex1c}
 	} \hspace{0.3cm}
 	\subfigure[$\mu=-1.7, a=2.3$]{
 		\includegraphics[height=2.5in,width=2.5in,keepaspectratio]{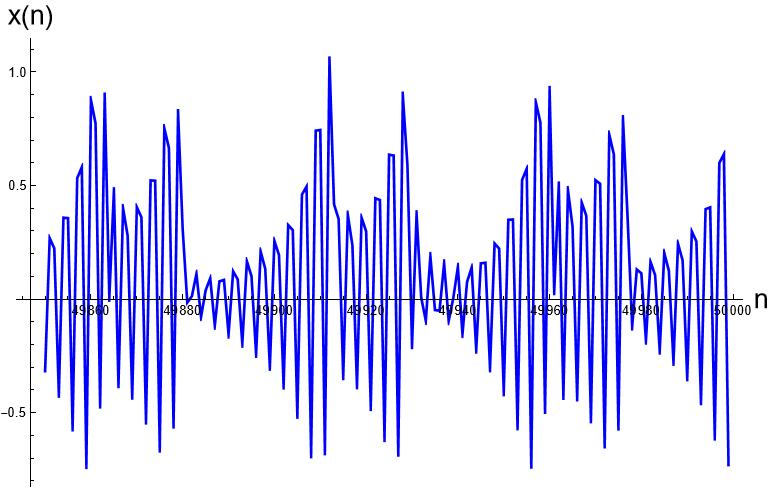}
 		\label{nonlinear_ex1d}
 	} \hspace{0.3cm}
 	\subfigure[$\mu=-1.1, a=3.5$]{
 		\includegraphics[height=2.5in,width=2.5in,keepaspectratio]{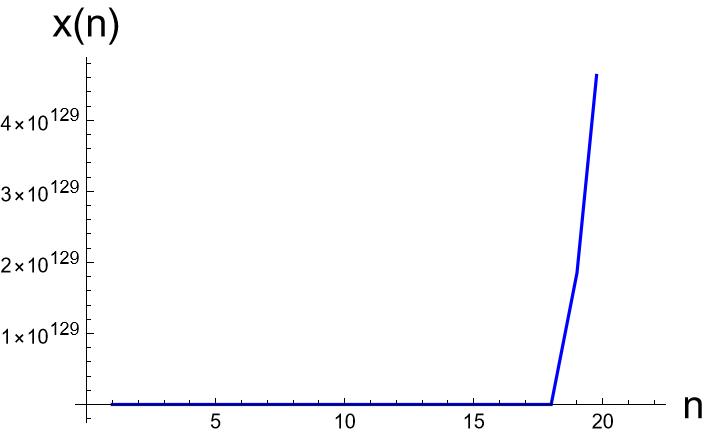}
 		\label{nonlinear_ex1e}
 	} \hspace{0.3cm}
 	\subfigure[$\mu=2.7, a=3.5$]{
 		\includegraphics[height=2.5in,width=2.5in,keepaspectratio]{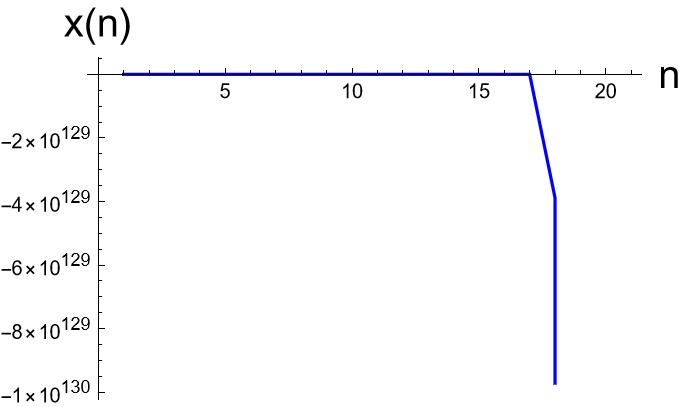}
 		\label{nonlinear_ex1f}
 	} 
	\caption{Solutions of system \eqref{nonlinear} for different values of $\mu$ and $a$ near the trivial equilibrium.}
 	\label{nonlinear_ex1}
 \end{figure}
 
 \begin{ex}
 Since $x_2^*$ depends on $\mu$, we set $x(0) = x_2^* + 0.1$ and $x(1) = x_2^* - 0.1$. For $a = 0.5$, Corollary \ref{cor:nonlinear} gives the linearized stability interval $\mu \in (1,2.284222)$. Because $f'(x_2^*)=2-\mu$, the choices $\mu=2.2$ and $\mu=3$ correspond to $w_{0.5}(-0.2)=2$ and $w_{0.5}(-1)=0$, respectively. Thus, the first trajectory converges slowly to $x_2^*=0.5455$ (Fig. \ref{nonlinear_ex1g}), whereas the second does not converge to $x_2^*=0.6667$ (Fig. \ref{nonlinear_ex1h}).
 
 \par For $a = 2.3$, the choices $\mu = -5$ and $\mu = 9$ give $f'(x_2^*)=7$ and $f'(x_2^*)=-7$, respectively. Both values have winding number zero and lie outside the linearized stability set; the corresponding solutions are unbounded (Figs. \ref{nonlinear_ex1i} and \ref{nonlinear_ex1j}). \label{NonlinearCaseEx2}
 \end{ex}
 \begin{figure}[H]
 	\centering
 	\subfigure[$\mu=2.2, a=0.5$]{
 		\includegraphics[height=2.5in,width=2.5in,keepaspectratio]{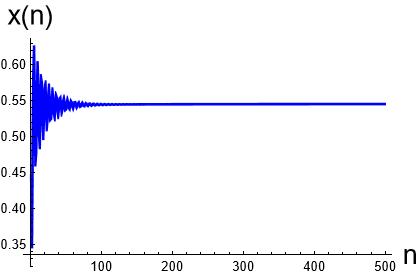}
 		\label{nonlinear_ex1g}
 	} \hspace{0.3cm}
 	\subfigure[$\mu=3, a=0.5$]{
 		\includegraphics[height=2.5in,width=2.5in,keepaspectratio]{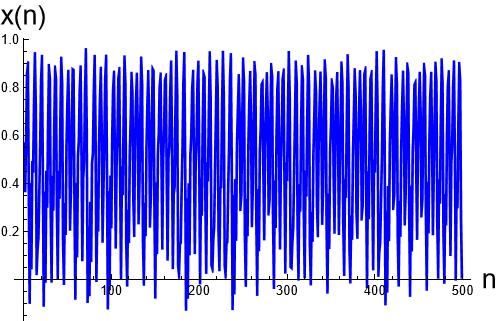}
 		\label{nonlinear_ex1h}
 	} \hspace{0.3cm}
 	\subfigure[$\mu=-5, a=2.3$]{
 		\includegraphics[height=2.5in,width=2.5in,keepaspectratio]{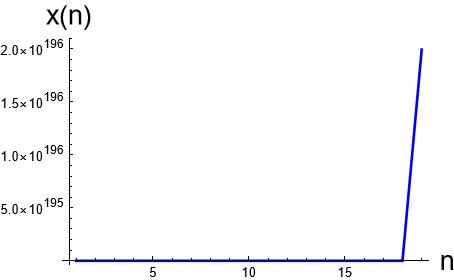}
 		\label{nonlinear_ex1i}
 	} \hspace{0.3cm}
 	\subfigure[$\mu=9, a=2.3$]{
 		\includegraphics[height=2.5in,width=2.5in,keepaspectratio]{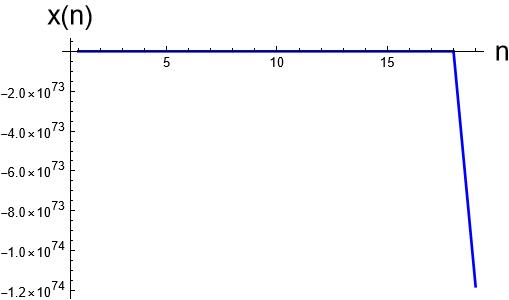}
 		\label{nonlinear_ex1j}
 	} 
	\caption{Solutions of system \eqref{nonlinear} for different values of $\mu$ and $a$ near the nontrivial equilibrium.}
 	\label{nonlinear_ex2}
 \end{figure}
  
\section{General higher-order case} \label{2model}

Consider the IVP
\begin{equation}
	\Delta^{\alpha}x(t) = (c - 1)x(t + \alpha - N), \label{Model}
\end{equation}
\[ x(i) = x_i, \qquad 0 \leq i \leq N - 1, \]
where $c \in \mathbb{C}$, $t \in \mathbb{N}_{N-\alpha}$, $N - 1 < \alpha \leq N$, and $N \in \mathbb{N}$.

The derivation below is written for $N-1<\alpha<N$. At the endpoint $\alpha=N$, the fractional operator is the ordinary $N$th forward difference, and the same characteristic equation follows directly.

By the definition of the Caputo-like difference operator, \eqref{Model} can be written as
\begin{equation}
	\left(\Delta^{-(N-\alpha)}\bigl(\Delta^N x\bigr)\right)(t)
	= (c - 1)x(t + \alpha - N), \label{def1}
\end{equation}
where $t = N - \alpha + n$ and $n \in \mathbb{N}_0$.

From the definition of a fractional sum, we have
\begin{equation}
	\left(\Delta^{-(N-\alpha)}y\right)(t)
	= \sum_{s=0}^{n} \binom{n - s + N - \alpha - 1}{n - s}y(s). \label{def2}
\end{equation}

Define the family of binomial functions on $\mathbb{Z}$, parameterized by $\mu \in \mathbb{R}$, by
\begin{equation}
\tilde{\phi}_{\mu}(n) =
\begin{cases}
	\displaystyle \binom{n + \mu - 1}{n}, & n \in \mathbb{N}_0, \\
	0, & n < 0.
\end{cases}
\end{equation}
Then \eqref{def2} can be written as
\begin{equation}
	\left(\Delta^{-(N-\alpha)}y\right)(t)
	= \sum_{s=0}^{n} \tilde{\phi}_{N-\alpha}(n - s)y(s). \label{def3}
\end{equation}
Consequently, \eqref{def1} becomes

\begin{equation}
	\sum_{s=0}^{n} \tilde{\phi}_{N-\alpha}(n - s)\Delta^N x(s) = (c - 1)x(n), \label{form}
\end{equation}
where $n = t + \alpha - N$.

Equation \eqref{form} can be expressed as the convolution of $\tilde{\phi}_{N-\alpha}$ and $\Delta^N x$:
\begin{equation}
	(\tilde{\phi}_{N-\alpha} \ast \Delta^N x)(n) = (c - 1)x(n), \label{convolution}
\end{equation}
where $\ast$ denotes convolution.

We have \cite{mozyrska2015z},
\begin{equation}
	\mathcal{Z}[\tilde{\phi}_{\alpha}](z) = \left(\frac{z}{z - 1}\right)^{\alpha}, \label{formula1}
\end{equation}

and \cite{elaydi2005introduction},
\begin{equation}
	\mathcal{Z}[\Delta^k x(n)] = (z - 1)^k X(z)
	- z\sum_{j=0}^{k-1}(z - 1)^{k-j-1}\Delta^j x(0). \label{formula2}
\end{equation}

Applying the Z-transform to \eqref{convolution} and using \eqref{formula1} and \eqref{formula2}, we obtain
\begin{equation}
	\left(\frac{z}{z - 1}\right)^{N-\alpha}
	\left((z - 1)^N X(z) - \text{terms involving the initial conditions}\right)
	= (c - 1)X(z). \label{in.val.term}
\end{equation}  
where $X(z)$ is the Z-transform of $x$.

Setting the coefficient of $X(z)$ in \eqref{in.val.term} equal to zero gives the characteristic equation of system \eqref{Model}:
\begin{equation}
	z^N(1 - z^{-1})^{\alpha} + 1 = c. \label{char.eqn}
\end{equation}
Substituting $z = e^{\mathrm{i}\theta}$ into \eqref{char.eqn} gives the complex boundary curve
\begin{equation}\label{Gamma-complex}
\Gamma_N(\theta)
=1+2^\alpha\left(\sin\frac{\theta}{2}\right)^\alpha
e^{\mathrm{i}\left[\frac{\alpha\pi}{2}+\theta\left(N-\frac{\alpha}{2}\right)\right]},
\qquad 0\leq\theta\leq2\pi.
\end{equation}
Consequently, the real parametrization is
\begin{align}
\Gamma_N(\theta)=\Biggl(&1+2^\alpha\left(\sin\frac{\theta}{2}\right)^\alpha
\cos\left(\frac{\alpha\pi}{2}+\theta\left(N-\frac{\alpha}{2}\right)\right), \nonumber\\
&2^\alpha\left(\sin\frac{\theta}{2}\right)^\alpha
\sin\left(\frac{\alpha\pi}{2}+\theta\left(N-\frac{\alpha}{2}\right)\right)\Biggr).
\label{Gammacurve}
\end{align}

\begin{thm}[General winding-number stability criterion]\label{thm:general-stability}
Let $N\in\mathbb{N}$ and $N-1<\alpha\leq N$. If $c\notin\Gamma_N([0,2\pi])$, then system \eqref{Model} is asymptotically stable if and only if
\begin{equation}\label{general-index}
\operatorname{Ind}_{\Gamma_N}(c)=N.
\end{equation}
\end{thm}

\begin{proof}
The characteristic function is
\begin{equation*}
D_c(z)=z^N(1-z^{-1})^\alpha+1-c.
\end{equation*}
As in the proof of Theorem \ref{thm:mixed-stability}, the principal branch of $D_c$ is analytic for $|z|>1$, and asymptotic stability is equivalent to the absence of characteristic zeros in the exterior of the closed unit disk. Applying the argument principle to $1+\varepsilon<|z|<R$ and letting $\varepsilon\downarrow0$ gives
\begin{equation*}
n_{\mathrm{out}}(c)
=\frac{1}{2\pi\mathrm{i}}\int_{|z|=R}\frac{D_c'(z)}{D_c(z)}\,dz
-\frac{1}{2\pi\mathrm{i}}\int_{|z|=1}\frac{D_c'(z)}{D_c(z)}\,dz.
\end{equation*}
On the unit circle, $D_c(e^{\mathrm{i}\theta})=\Gamma_N(\theta)-c$, so the inner-boundary integral is $\operatorname{Ind}_{\Gamma_N}(c)$. Since $D_c(z)/z^N\to1$ uniformly as $R\to\infty$, the outer-boundary integral equals $N$ for sufficiently large $R$. Hence,
\begin{equation*}
n_{\mathrm{out}}(c)=N-\operatorname{Ind}_{\Gamma_N}(c).
\end{equation*}
There are no characteristic zeros on $|z|=1$ because $c\notin\Gamma_N([0,2\pi])$. Thus, asymptotic stability is equivalent to $n_{\mathrm{out}}(c)=0$, which proves \eqref{general-index}.
\end{proof}

\begin{Lem}[Winding-number bound]\label{lem:winding-bound}
Set $q=N-\alpha/2$. For every $c\notin\Gamma_N([0,2\pi])$,
\begin{equation}\label{index-bound}
\left|\operatorname{Ind}_{\Gamma_N}(c)\right|\leq\lceil q\rceil.
\end{equation}
\end{Lem}

\begin{proof}
From \eqref{Gammacurve}, write
\begin{equation*}
\Gamma_N(\theta)
=1+r(\theta)e^{\mathrm{i}(\alpha\pi/2+q\theta)},
\qquad
r(\theta)=2^\alpha\left(\sin\frac{\theta}{2}\right)^\alpha.
\end{equation*}
Also write $1-c=\rho e^{\mathrm{i}\varphi}$. Since $c=1$ belongs to the boundary curve, $c\notin\Gamma_N([0,2\pi])$ implies $\rho>0$. We translate the curve by $-c$ so that the point about which the winding number is computed becomes the origin, and then rotate it through the fixed angle $-\varphi$. The resulting curve is
\begin{equation*}
\begin{split}
u(\theta)
&=e^{-\mathrm{i}\varphi}\bigl(\Gamma_N(\theta)-c\bigr)\\
&=\rho+r(\theta)e^{\mathrm{i}\chi(\theta)},
\end{split}
\end{equation*}
where
\begin{equation*}
\chi(\theta)=\frac{\alpha\pi}{2}-\varphi+q\theta.
\end{equation*}
Translation gives
\begin{equation*}
\operatorname{Ind}_{\Gamma_N}(c)
=\operatorname{Ind}_{\Gamma_N-c}(0),
\end{equation*}
and multiplication by $e^{-\mathrm{i}\varphi}$ is only a rotation about the origin, so it does not change the winding number. Consequently,
\begin{equation*}
\operatorname{Ind}_{\Gamma_N}(c)=\operatorname{Ind}_{u}(0).
\end{equation*}
The purpose of this normalization is that the constant term $1-c$ becomes the positive real number $\rho$, which makes the crossings of the negative real axis easy to count.

The winding number of a closed curve about the origin is the signed number
of times the curve crosses any fixed ray from the origin. We use the
negative real axis. If $u(\theta)$ lies on this ray, then
\begin{equation*}
\operatorname{Im}u(\theta)=r(\theta)\sin\chi(\theta)=0
\end{equation*}
and its real part is negative. Since $r(\theta)>0$ for
$0<\theta<2\pi$, this is possible only when
\begin{equation*}
\chi(\theta)=(2k+1)\pi
\end{equation*}
for some integer $k$; the additional condition $r(\theta)>\rho$ only
reduces the number of crossings. The function $\chi$ is strictly
increasing and
\begin{equation*}
\chi(2\pi)-\chi(0)=2\pi q.
\end{equation*}
Successive odd multiples of $\pi$ are separated by $2\pi$. Therefore an
interval of length $2\pi q$ contains at most $\lceil q\rceil$ such
values. Hence the curve crosses the negative real axis at most
$\lceil q\rceil$ times. The absolute value of the signed crossing number
cannot exceed the total number of crossings, and thus
\begin{equation*}
\left|\operatorname{Ind}_{\Gamma_N}(c)\right|
=\left|\operatorname{Ind}_{u}(0)\right|
\leq\lceil q\rceil.
\end{equation*}
Tangential contacts, if present, have zero signed contribution and do not
affect the estimate.
\end{proof}

\begin{thm}[Absence of stability for $N\geq3$]\label{thm:no-stability}
Let $N\geq3$ and $N-1<\alpha\leq N$. Then the stability region of system \eqref{Model} is empty; that is, the zero solution is not asymptotically stable for any $c\in\mathbb{C}$.
\end{thm}

\begin{proof}
Because $N-1<\alpha\leq N$,
\begin{equation*}
q=N-\frac{\alpha}{2}<\frac{N+1}{2}.
\end{equation*}
For every integer $N\geq3$, this implies $\lceil q\rceil\leq N-1$. Lemma \ref{lem:winding-bound} therefore gives
\begin{equation*}
\operatorname{Ind}_{\Gamma_N}(c)\leq N-1
\end{equation*}
for every $c$ outside the boundary. The necessary and sufficient condition \eqref{general-index} can never be satisfied. Parameters on the boundary have a characteristic zero with $|z|=1$ and are not asymptotically stable either. Hence no $c\in\mathbb{C}$ yields asymptotic stability.
\end{proof}
For $N=1$, the IVP becomes
\begin{equation}
	\Delta^{\alpha}x(t) = (c - 1)x(t + \alpha - 1), \label{Model1}
\end{equation}
\[ x(0) = x_0. \]
Here, $0 < \alpha \leq 1$. Figure \ref{N1} shows stability regions in the complex plane for several values of $\alpha$. The stability region grows as $\alpha$ increases.
\par For $c \in \mathbb{R}$, substituting $N = 1$ and $z = \pm 1$ into \eqref{char.eqn} gives the stability interval $(1 - 2^\alpha,1)$ on the real axis.
\begin{figure}[H]
	\centering
	\includegraphics[height=4.1in,width=4.1in,keepaspectratio]{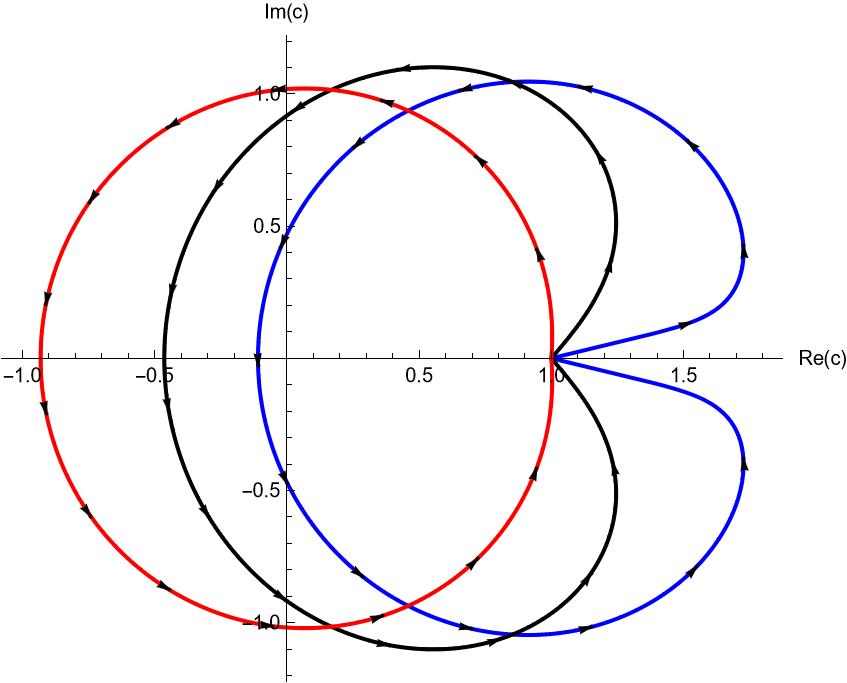}
	\caption{Stability boundaries of system \eqref{Model1} for $\alpha = 0.15$ (blue), $\alpha = 0.55$ (black), and $\alpha = 0.95$ (red). For each oriented curve, its stable component is characterized by $\operatorname{Ind}_{\Gamma_1}(c)=1$; the unbounded exterior has index zero.}
	\label{N1}
\end{figure}

\begin{ex}
	Let $\alpha = 0.55$ and $x(0) = 0.4$. Numerical evaluation of the winding number gives $\operatorname{Ind}_{\Gamma_1}(0.982 + 0.4906\mathrm{i})=1=N$, so the solution of system \eqref{Model1} converges to zero (Fig. \ref{N1ex1a}). In contrast, $\operatorname{Ind}_{\Gamma_1}(0.1346 - 1.101\mathrm{i})=0$, and the corresponding solution diverges from zero (Fig. \ref{N1ex1b}).
\end{ex}
\begin{figure}[H]
	 \centering
	\subfigure[$c=0.982+0.4906\mathrm{i}$, $\operatorname{Ind}_{\Gamma_1}=1$.]{
		\includegraphics[height=2.5in,width=2.5in,keepaspectratio]{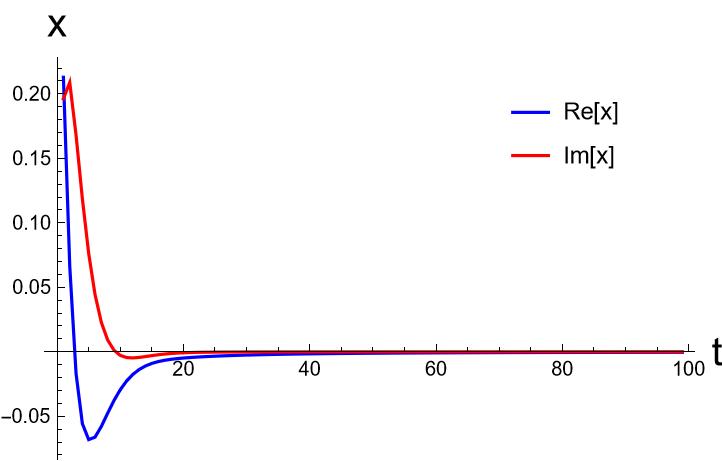}
		\label{N1ex1a}
		} \hspace{0.3cm}
	\subfigure[$c=0.1346-1.101\mathrm{i}$, $\operatorname{Ind}_{\Gamma_1}=0$.]{
		\includegraphics[height=2.5in,width=2.5in,keepaspectratio]{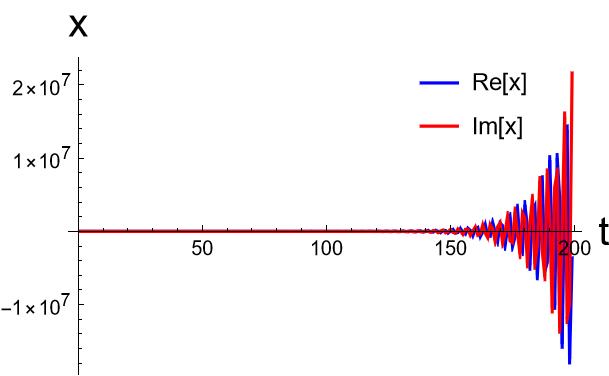}		\label{N1ex1b}}
	\caption{Solutions of system \eqref{Model1} for different values of $c$.}
	\label{N1ex1}
\end{figure}
For $N=2$, the IVP becomes
\begin{equation}
	\Delta^{\alpha}x(t) = (c - 1)x(t + \alpha - 2), \label{Model2}
\end{equation}
\[ x(0) = x_0, \qquad x(1) = x_1. \]
Here, $1 < \alpha \leq 2$. For each $\alpha \in (1,2)$, the curve $\Gamma_2(\theta)$ is bounded in the complex plane and contains a positively oriented closed loop near $c = 1$ (Fig. \ref{N2a}). The component enclosed by this loop has winding number two and is therefore stable. In the displayed cases, this component shrinks as $\alpha$ approaches $2$. Figure \ref{N2abc} shows the corresponding stability regions.

\begin{figure}[H]
	\centering
	\subfigure[$\alpha=1.1$.]{
		\includegraphics[height=2.7in,width=2.7in,keepaspectratio]{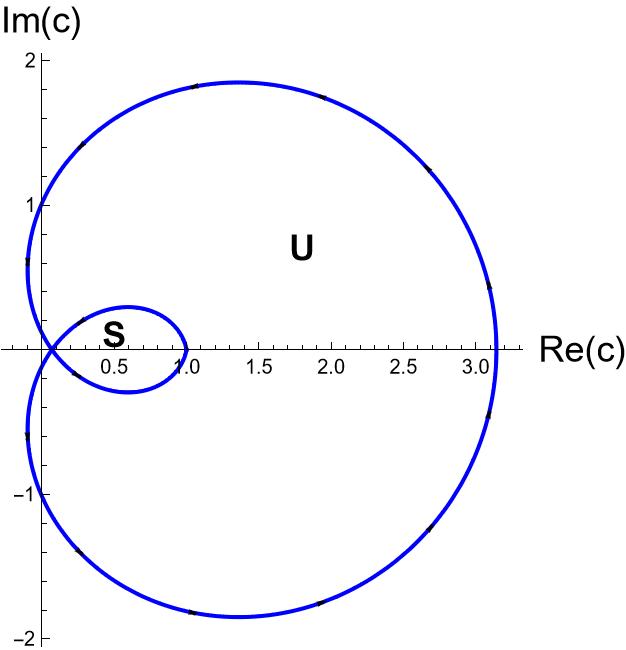}
		\label{N2a}
	} \hspace{0.3cm}
	\subfigure[$\alpha=1.5$.]{
		\includegraphics[height=2.7in,width=2.7in,keepaspectratio]{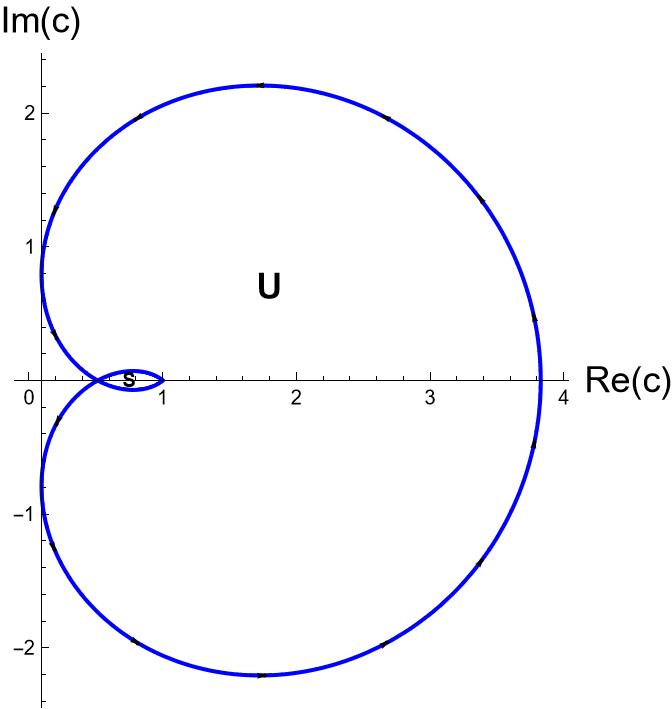}	\label{N2b}
	} \hspace{0.3cm}
	\subfigure[$\alpha=1.8$.]{
	\includegraphics[height=2.7in,width=2.7in,keepaspectratio]{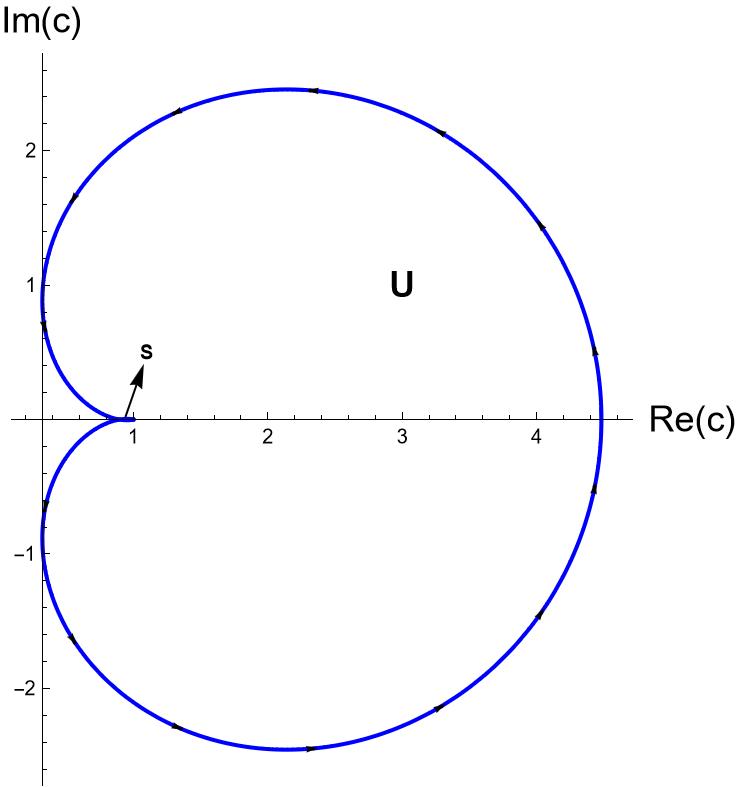}		\label{N2c}
	}
	\caption{Stability regions of system \eqref{Model2} for different values of $\alpha$. In every panel, the component marked S has $\operatorname{Ind}_{\Gamma_2}(c)=2$, whereas the component marked U has an index different from two.}
	\label{N2abc}
\end{figure}
\begin{ex}
	Let $\alpha = 1.1$, $x(0) = 0.1$, and $x(1) = 0.2$. We obtain $\operatorname{Ind}_{\Gamma_2}(0.2415 - 0.06215\mathrm{i})=2=N$, and the solution of system \eqref{Model2} converges to zero (Fig. \ref{N2ex1a}). For $c = 0.1346 - 0.8733\mathrm{i}$, the winding number is one rather than two, and the solution is unbounded (Fig. \ref{N2ex1b}).
\end{ex}
\begin{figure}[H]
	\centering
	\subfigure[$c=0.2415-0.06215\mathrm{i}$, $\operatorname{Ind}_{\Gamma_2}=2$.]{
		\includegraphics[height=2.5in,width=2.5in,keepaspectratio]{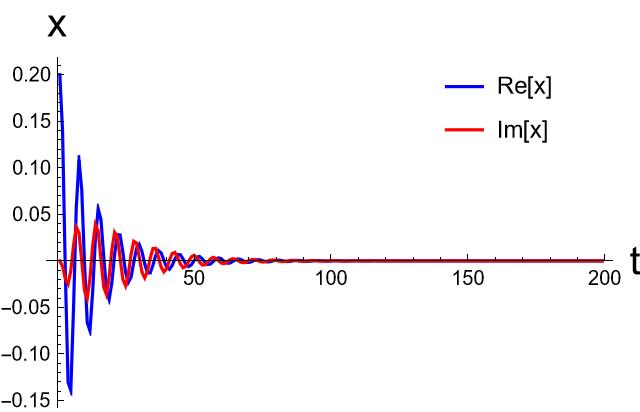}
		\label{N2ex1a}
	} \hspace{0.3cm}
	\subfigure[$c=0.1346-0.8733\mathrm{i}$, $\operatorname{Ind}_{\Gamma_2}=1$.]{
		\includegraphics[height=2.5in,width=2.5in,keepaspectratio]{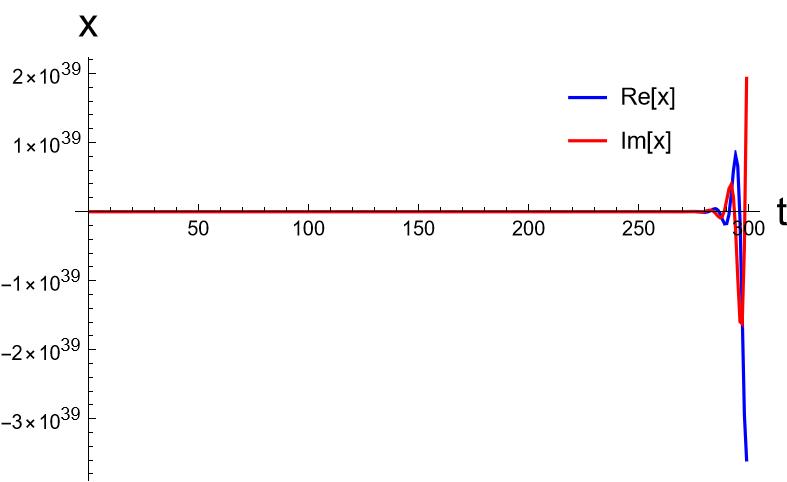}		\label{N2ex1b}}
	\caption{Solutions of system \eqref{Model2} for different values of $c$.}
	\label{N2ex1}
\end{figure}
Theorem \ref{thm:no-stability} proves that the stability region is empty for every $N \geq 3$. The following numerical experiments illustrate this result.
For $N=3$, the IVP becomes 
\begin{equation}
	\Delta^{\alpha}x(t) = (c - 1)x(t + \alpha - 3), \label{Model3}
\end{equation}
\[ x(0) = x_0, \qquad x(1) = x_1, \qquad x(2) = x_2, \]
where $2 < \alpha < 3$.
\begin{ex}
	Let $\alpha = 2.5$, $x(0) = 0.01$, $x(1) = 0.02$, and $x(2) = 0.03$. For $c = 0.3839 + 4.832\mathrm{i}$, $c = 0.6667 - 0.6024\mathrm{i}$, $c = -4.02 - 3.168\mathrm{i}$, and $c = -2.168 + 1.798\mathrm{i}$, the winding numbers are $0$, $1$, $0$, and $1$, respectively. None equals the required value $N=3$, and all four solutions of system \eqref{Model3} are unbounded (Fig. \ref{N3ex}).
\end{ex}
\begin{figure}[H]
	\centering
	\subfigure[$c = 0.3839 + 4.832\mathrm{i}$.]{
		\includegraphics[height=2.5in,width=2.5in,keepaspectratio]{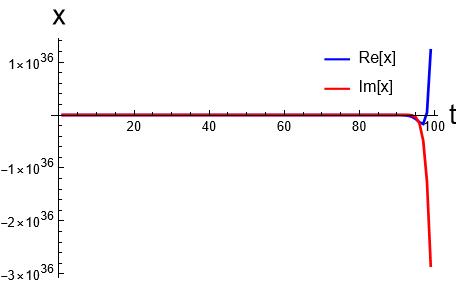}
		\label{N3ex1a}
	} \hspace{0.3cm}\subfigure[$c = 0.6667 - 0.6024\mathrm{i}$.]{
	\includegraphics[height=2.5in,width=2.5in,keepaspectratio]{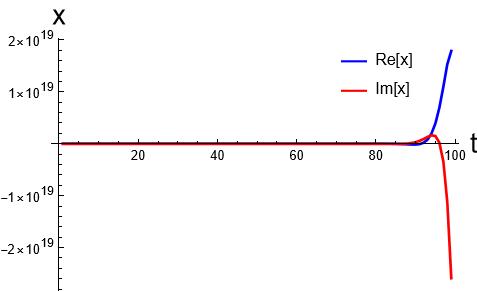}
	\label{N3ex1b}
	} \hspace{0.3cm}\subfigure[$c = -4.02 - 3.168\mathrm{i}$.]{
	\includegraphics[height=2.5in,width=2.5in,keepaspectratio]{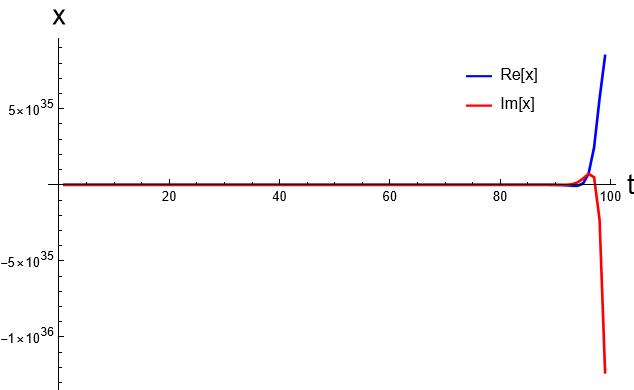}
	\label{N3ex1c}
	} \hspace{0.3cm}
	\subfigure[$c = -2.168 + 1.798\mathrm{i}$.]{
		\includegraphics[height=2.5in,width=2.5in,keepaspectratio]{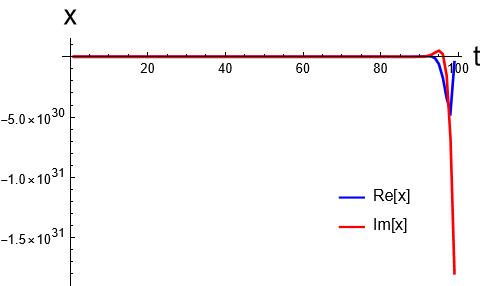}		\label{N3ex1d}}
	\caption{Solutions of system \eqref{Model3} for different values of $c$.}
	\label{N3ex}
\end{figure}
 
 For $N = 6$, the IVP becomes
 \begin{equation}
	\Delta^{\alpha}x(t) = (c - 1)x(t + \alpha - 6), \label{Model6}
 \end{equation}
 \[ x(0) = x_0, \quad x(1) = x_1, \quad x(2) = x_2, \quad x(3) = x_3, \quad x(4) = x_4, \quad x(5) = x_5, \]
 where $5 < \alpha < 6$.
\begin{ex}
	Let $\alpha = 5.3$, with initial conditions $x(0) = 0.1$, $x(1) = 0.2$, $x(2) = 0.3$, $x(3) = 0.4$, $x(4) = 0.5$, and $x(5) = 0.6$. For $c = 1.948 + 0.06482\mathrm{i}$, $c = -4.692 + 2.35\mathrm{i}$, $c = -7.879 - 23.47\mathrm{i}$, and $c = 1.517 - 0.3911\mathrm{i}$, the winding numbers are $3$, $2$, $1$, and $3$, respectively. Since none equals $N=6$, the solutions of system \eqref{Model6} are unbounded (Fig. \ref{N6ex}).
\end{ex}
 \begin{figure}[H]
 	\centering
	\subfigure[$c = 1.948 + 0.06482\mathrm{i}$.]{
 		\includegraphics[height=2.5in,width=2.5in,keepaspectratio]{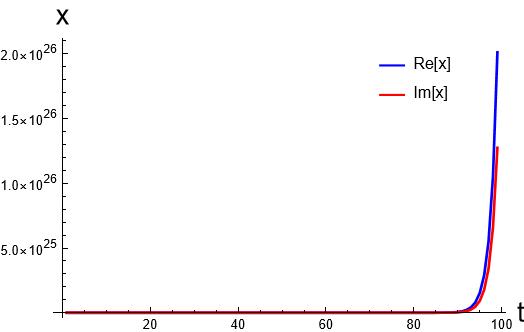}
 		\label{N6ex1a}
	} \hspace{0.3cm}\subfigure[$c = -4.692 + 2.35\mathrm{i}$.]{
 		\includegraphics[height=2.5in,width=2.5in,keepaspectratio]{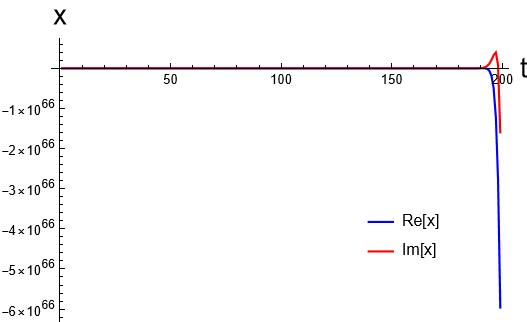}
 		\label{N6ex1b}
	} \hspace{0.3cm}\subfigure[$c = -7.879 - 23.47\mathrm{i}$.]{
 		\includegraphics[height=2.5in,width=2.5in,keepaspectratio]{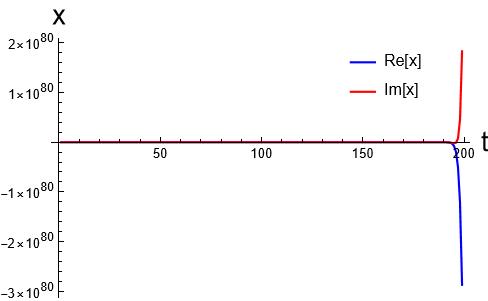}
 		\label{N6ex1c}
 	} \hspace{0.3cm}
	\subfigure[$c = 1.517 - 0.3911\mathrm{i}$.]{
 		\includegraphics[height=2.5in,width=2.5in,keepaspectratio]{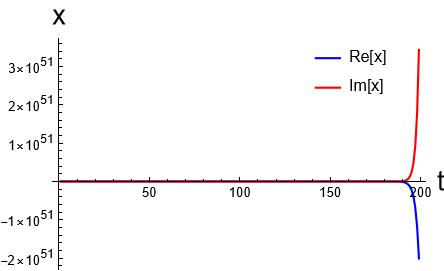}		\label{N6ex1d}}
	\caption{Solutions of system \eqref{Model6} for different values of $c$.}
 	\label{N6ex}
 \end{figure}
\section{Conclusion} \label{con.}
We studied the two-term fractional difference equation $\Delta^{\alpha}x(t) + a\,\Delta^{\beta}x(t + \alpha - \beta - 1) = (b - 1)x(t + \alpha - 2)$ for $0 < \beta \leq 1 < \alpha \leq 2$. The Z-transform yields a characteristic boundary whose winding number gives a necessary and sufficient asymptotic-stability condition. We proved that $a_1=2^{\alpha-\beta}$ is the unique value for which $\gamma(0)=\gamma(\pi)$ and that $a_2=2^{\alpha-\beta}(4-\alpha)/(2-\beta)$ is the unique value for which the boundary is stationary at $\theta=\pi$. The numerical examples show how the index-two stability component changes at these values. The linear criterion also describes the linearized behavior of the equilibria of the nonlinear logistic model. For the arbitrary-order family $\Delta^{\alpha}x(t) = (c - 1)x(t + \alpha - N)$, the winding-number bound proves that the stability region is empty for every $N\geq3$. The numerical experiments are consistent with these theoretical results.
\section*{Funding}
Ch. Janardhan acknowledges financial support from the University Grants Commission, New Delhi, India, under grant No. F.14-34/2011(CPP-II).

\section*{Declarations}
The authors declare that they have no known competing financial interests or personal relationships that could have appeared to influence the work reported in this paper.

Generative AI tools were used only to improve the language and clarity of this manuscript. The authors carefully reviewed and edited the output and take full responsibility for the final content.


\section*{Data availability statement}
No external data sets were used in this study. All numerical results reported in the article were generated from the equations and parameter values specified in the text.

\bibliography{ref.bib}
\end{document}